\documentclass[a4paper,11pt]{article}
\usepackage[top=3.5cm,bottom=3.5cm,right=3cm,left=3cm]{geometry}
\usepackage[T1]{fontenc}
\usepackage[utf8]{inputenc}

\usepackage{amsfonts}
\usepackage{amsmath}
\usepackage{amssymb}
\usepackage{amscd}
\usepackage{amsthm}
\usepackage{mathrsfs}

\usepackage{makeidx}
\makeindex

\usepackage{lettrine}
\usepackage{yfonts}

\usepackage[french]{babel}
\usepackage[all]{xy}

\newtheorem{theorem}{Th\'eor\`eme}[subsection]
\newtheorem{lemma}{Lemme}[subsection]
\newtheorem{proposition}{Proposition}[subsection]

\newtheorem{cor}{Corollaire}

\newtheorem{defi}{D\'efinition}

\begin{document}

\date{}
\title{\large{LE SYMBOLE DE HASSE LOGARITHMIQUE}}

\author{{\bf Stéphanie Reglade}}

\maketitle

{\footnotesize \noindent\textbf{Résumé}: \textit{Le but de cet article est de présenter la construction du symbole de Hasse logarithmique, analogue du symbole classique dans le contexte logarithmique et d'en étudier les principales propriétés.
Cette étude permet d'obtenir une expression du défaut du principe de Hasse $\ell$-adique.
Nous présentons ensuite des cas particuliers et
obtenons une version logarithmique du théorème de l'idéal principal.}}
\smallskip

{\footnotesize \noindent\textbf{Abstract}: \textit{In this article, we define the logarithmic Hasse symbol in the same way as
the usual one but in the context of the logarithmic ramification. We study its fondamental properties. The interesting point
is that we get an expression of the defect of the $\ell$-adic Hasse principle. Then we study particular cases and get a logarithmic version of the principal ideal theorem.      }}

\medskip

\noindent \textbf{Mots-clefs}: \textbf{théorie du corps des classes, théorie $\ell$-adique du corps des classes, théorie logarithmique.}
\medskip

\medskip

\tableofcontents

\bigskip

\section{Introduction et Notations}

\subsection{Position du problème}

\noindent Les objets étudiés sont les $\mathbb{Z}_{\ell}$-modules fondamentaux de la théorie $\ell$-adique 
du corps des classes construits par Jaulent dans \cite{Ja1} où $\ell$ désigne un nombre premier fixé. Le contexte logarithmique
correspond au cas où l'objet local étudié : le $\ell$-adifié du groupe multiplicatif d'un corps de nombres
est muni de la valuation dite logarithmique, construite à partir du logarithme d'Iwasawa \cite{Ja2}.
\smallskip

\noindent  Le théorème $\ell$-adique de la norme de Hasse \cite{Re1} précise
que dans le cas d'une $\ell$-extension cyclique un idèle principal est une norme globale si et seulement si
c'est une norme locale partout  i.e une norme pour chaque complétion $ L_{\mathfrak{P}}/K_{\mathfrak{p}}$ où $ \mathfrak{P} \mid \mathfrak{p}$. 

\noindent Nous sommes donc spontanément amenés à considérer, dans le cas d'une $\ell$-extension
abélienne finie, le quotient des idèles principaux normes locales partout par ceux  qui sont normes globales : ce quotient définit le "groupe de défaut" du principe de Hasse $\ell$-adique, objet analogue au groupe de défaut du principe de Hasse classique.
\smallskip

\noindent Le but de cet article est d'étudier ce groupe de défaut du principe de Hasse $\ell$-adique et d'en donner
 une interprétation arithmétique  le reliant au groupe des classes logarithmique de degré nul. C'est le sens du théorème 3.0.4 :

\medskip

\noindent \textbf{Théorème}

\noindent \textit{Soit $L/K$ une $\ell$-extension abélienne finie}, $\tilde{\mathcal{C}}\ell_{L} $ (\textit{respectivement} $\tilde{\mathcal{C}}\ell_{K} $  ) \textit{le groupe des classes logarithmiques de degré nul de} $L$ (\textit{respectivement} $K$),  $\tilde{\mathcal{C}}\ell_{L}^{*}$  \textit{le noyau de l'application norme} $N_{L/K} : \tilde{\mathcal{C}}\ell_{L} 
\longrightarrow N_{L/K} \tilde{\mathcal{C}}\ell_{L}  $  \textit{et} $\Delta_{L/K}$ \textit{ l'idéal
d'augmentation du groupe de Galois de} $L/K$. \textit{Nous avons alors :}

$$ | \hat{\Gamma}_{L/K}|  (\mathcal{N}_{L/K} :  N_{L/K}\mathcal{R}_{L} ) =(\tilde{\mathcal{C}\ell}_{L}^{*}: \tilde{\mathcal{C}\ell}_{L}^{\Delta_{L/K}}) 
(\tilde{\mathcal{E}}_{K}: \tilde{\mathcal{E}}_{K} \cap  N_{L/K}\mathcal{R}_{L}).$$

\medskip

\noindent L'outil utilisé pour démontrer ce théorème est le symbole de Hasse logarithmique défini dans le paragraphe suivant. 
Nous explicitons ensuite ces principales propriétés et établissons une formule analogue à la formule
du produit pour ce symbole. 
\medskip

\noindent Nous étudions par la suite un exemple d'application et des cas particuliers de ce théorème.
Pour finir nous en présentons un corollaire : une version logarithmique du théorème de l'idéal principal.

\bigskip\bigskip

\subsection{Glossaire des notations}

\noindent{Dans tout ce qui suit $\ell$ désigne un nombre premier fixé. Introduisons les notations suivantes}
\medskip

\noindent{\textit{Pour un corps local $K_{\mathfrak{p}}$ d'idéal maximal $\mathfrak{p}$ et d'uniformisante $\pi_{\mathfrak{p}}$, nous notons}
\bigskip

$\mathcal{R}_{K_{\mathfrak{p}} }=\varprojlim_{k}  K_{\mathfrak{p}}^{\times} \diagup {  K_{\mathfrak{p}}^{\times  \ell^{k}}}$:  
le $\ell$-adifié du groupe multiplicatif du corps local
\medskip

$\mathcal{U}_{K_{\mathfrak{p}} }=\varprojlim_{k} {U}_{\mathfrak{p}}  \diagup  U_{\mathfrak{p}}^{\ell^k}$: le $\ell$-adifié  du groupe des unités  $U_{\mathfrak{p}} $  de $K_{\mathfrak{p}}$
\medskip

\medskip

\noindent{\textit{Pour un corps de nombre $K$, nous définissons }}
\bigskip

$\mathcal{R}_{K}=\mathbb{Z}_{\ell} \otimes_{\mathbb{Z}} K^{\times}$ : le
$\ell$-groupe des idèles principaux
\medskip

 $\mathcal{J}_{K} =\prod_{ \mathfrak{p} \in Pl_{K} } ^{res} \mathcal{R}_{K_{\mathfrak{p}}}$ : le $\ell$-groupe des idèles 
\medskip

$\mathcal{U}_{K}=\prod_{ \mathfrak{p} \in Pl_{K} } \mathcal{U}_{K_{\mathfrak{p}}}$ 
: le sous-groupe des unités
\medskip

$\mathcal{C}_{K}= \mathcal{J}_{K} /  \mathcal{R}_{K}$ : le $\ell$-groupe des classes d'idèles
\bigskip

\noindent{\textit{Dans le contexte logarithmique, nous posons }}
\bigskip

$\hat{\mathbb{Q}_{\mathfrak{p}}^{c}}$ : la $\widehat{\mathbb{Z}}$-extension cyclotomique de  $\mathbb{Q}_{p}$
\medskip

$\mathbb{Q}_{p}^{c}$ : la $\mathbb{Z}_{\ell}$-extension cyclotomique de $\mathbb{Q}_{p}$
\medskip

$\tilde v_{\mathfrak{p}}$ : la valuation logarithmique  associée à $\mathfrak{p}$ sur $\mathcal{R}_{K_{\mathfrak{p}}}$
\medskip

$\widetilde{\mathcal{U}}_{ K_{\mathfrak{p}}}=\textrm{Ker}(\tilde{v_{\mathfrak{p}}})$ : le sous groupe des unités logarithmiques locales
\medskip

$\widetilde{\mathcal{U}}_{K}=\prod_{ \mathfrak{p} \in Pl_{K} } \widetilde{\mathcal{U}}_{K_{\mathfrak{p}}}$ 
: le sous-groupe des unités logarithmiques
\medskip

$\tilde e_{\mathfrak{p}} = [ K_{\mathfrak{p}} : \hat{\mathbb{Q}_{\mathfrak{p}}^{c}} \cap K_{\mathfrak{p}} ]$ :
l'indice absolu de ramification logarithmique de $\mathfrak{p}$
\medskip

$ \tilde f_{\mathfrak{p}}=[ \hat{\mathbb{Q}_{\mathfrak{p}}^{c}} \cap K_{\mathfrak{p}} : \mathbb{Q}_{\mathfrak{p}}]$ :
le degré absolu d'inertie logarithmique de $\mathfrak{p}$
\medskip

$ \tilde e_{ L_{\mathfrak{P}}/K_{\mathfrak{p}}} = [ L_{\mathfrak{P}} :
\hat{K_{\mathfrak{p}}^{c}} \cap L_{\mathfrak{P}} ]$ : l'indice relatif de ramification logarithmique de $\mathfrak{p}$
\medskip

$ \tilde f_{ L_{\mathfrak{P}}/K_{\mathfrak{p}}}=[ \widehat{K_{\mathfrak{p}}^{c}} \cap L_{\mathfrak{P}} : K_{\mathfrak{p}}]$ :
le degré relatif d'inertie logarithmique de $\mathfrak{p}$
\medskip

 $\mathcal{J}_{K}^{( \mathfrak{m})}=\prod_{\mathfrak{p} \not \vert  \mathfrak{m} } \mathcal{R}_{K_{\mathfrak{p}}} \prod_{\mathfrak{p} \vert  \mathfrak{m} } \widetilde{\mathcal{U}}_{K_{\mathfrak{p}}}^{v_{\mathfrak{p}}( \mathfrak{m})}$
\medskip

$\mathcal{J}_{K}^{ \mathfrak{m}}=\prod_{\mathfrak{p} \not \vert  \mathfrak{m} } \mathcal{R}_{K_{\mathfrak{p}}} \prod_{\mathfrak{p} \vert  \mathfrak{m} }
\widetilde{\mathcal{U}}_{K_{\mathfrak{p}}} $
\medskip

$\widetilde{\mathcal{U}}_{K}^{( \mathfrak{m})}=\prod_{\mathfrak{p} \not \vert  \mathfrak{m} } \mathcal{U}_{K_{\mathfrak{p}}} \prod_{\mathfrak{p} \vert  \mathfrak{m} } \widetilde{\mathcal{U}}_{K_{\mathfrak{p}}}^{v_{\mathfrak{p}}( \mathfrak{m})} $
\medskip

$\mathcal{R}_{K}^{( \mathfrak{m})}=\mathcal{R}_{K} \cap \mathcal{J}_{K}^{( \mathfrak{m})}$
\medskip

$D\ell_{K}=\bigoplus_{\mathfrak{p} \; \textrm{place finie de $K$}}  \mathbb{Z}_{\ell} \mathfrak{p}$ : le groupe des diviseurs logarithmiques de $K$
\medskip

\begin{align*}
\psi : \mathcal{J}_{K}^{ \mathfrak{m}} &\longrightarrow  D\ell_{K}^{ \mathfrak{m}} \\
           \alpha=(\alpha_{\mathfrak{p}}) &\longmapsto \psi(\alpha)=\prod_{\textrm{place finie de $K$}}
\mathfrak{p}^{\widetilde{v_{\mathfrak{p}}}(\alpha_{\mathfrak{p}})}
\end{align*}
\medskip
           
$ D\ell_{K}^{ \mathfrak{m}}=\psi(\mathcal{J}_{K}^{ \mathfrak{m}})$ : diviseurs logarithmiques premiers à $m$ 
\medskip

$ P\ell_{K}^{( \mathfrak{m})}=\psi(\mathcal{R}_{K}^{( \mathfrak{m})})$ : diviseurs logarithmiques principaux attachés à $ \mathfrak{m}$
\medskip

$\tilde{D\ell}_{L/K}$ : groupe des diviseurs logarithmiques de degré nul
\medskip

$\tilde{\mathcal{C}\ell}_{L/K}$ : groupe des classes logarithmiques de degré nul

$\widetilde{\mathfrak{f}}_{L/K}$ : le conducteur logarithmique global de $L/K$
\medskip

$\widetilde{\mathfrak{f}}_{\mathfrak{p}} $ : le conducteur logarithmique local en $\mathfrak{p}$

$(\widetilde{\frac{L/K}{\mathfrak{p}} })$ : le Frobenius logarithmique en $\mathfrak{p}$
\medskip

$A\ell_{L/K}$ le groupe d'Artin logarithmique de $L/K$
\medskip

$ T\ell_{K,m} = P\ell_{K}^{( \mathfrak{m})} \cdot N_{L/K}( D\ell_{L}^{ \mathfrak{m}} ))$ : le sous-module de Takagi logarithmique associé à $ \mathfrak{m}$
\medskip

$\widetilde{ ( \frac{\alpha,L/K} { \mathfrak{p}})}$ : le symbole de Hasse logarithmique
\medskip

$\widetilde{\Gamma}_{L/K, \mathfrak{p}}$ : le sous-groupe d'inertie logarithmique en $\mathfrak{p}$
\medskip

$\widetilde{\Gamma}_{L/K}= \prod_{\mathfrak{p} \vert \tilde{\mathfrak{f}}_{L/K} } \widetilde{\Gamma}_{L/K, \mathfrak{p}} $
\medskip

$ \mathcal{A}_{L/K} =\{ \alpha \in \mathcal{R}_{K} \; / \; \psi(\alpha) \in N_{L/K}(D\ell_{L}^{\tilde{\mathfrak{f}}_{L/K}}) \}.$
\medskip

$\Lambda_{ \mathfrak{m}}$ : idèles principaux dont le diviseur logarithmique associé est premier à $m$
\medskip

$ \mathcal{N}_{L/K}= \{ \alpha \in \mathcal{R}_{K} \; \textrm{ qui sont partout norme locale} \}. $
\medskip

$ \mathcal{N}_{L/K, \mathfrak{m}}= \mathcal{N}_{L/K} \cap \Lambda_{ \mathfrak{m}}. $ pour tout module $ \mathfrak{m}$ de $K$
\medskip

\section{Le symbole de Hasse logarithmique}
\medskip

\noindent Nous introduisons, dans cette partie, le symbole de Hasse logarithmique analogue du symbole de Hasse
classique dans le  contexte logarithmique. Nous nous intéressons à l'image et au noyau de cet homomorphisme.
Nous rappelons au préalable les définitions de conducteur logarithmique et de l'application d'Artin logarithmique.

\medskip

\subsection{Définitions et premières propriétés}

\noindent L'objet local de la théorie $\ell$-adique du corps des classes \cite{Ja1} est le $\ell$-adifié du groupe multiplicatif d'un corps local : $\mathcal{R}_{K_{\mathfrak{p}} }=\varprojlim_{k}  K_{\mathfrak{p}}^{\times} \diagup {  K_{\mathfrak{p}}^{\times  \ell^{k}}}.$  Il est ici muni de la valuation logarithmique définie comme suit \cite{Ja2} :

\begin{defi}{\cite{Ja2}}
Soit $K$ une extension finie de $\mathbb{Q}$, $p$ un nombre premier. Notons $\widehat{\mathbb{Q}_{p}^{c}}$ la $\widehat{\mathbb{Z}}$-extension  cyclotomique de $\mathbb{Q}_{p}$. Soit $\mathfrak{p}$ une place de $K$ au dessus de $p$,

\noindent i) le degré $\ell$-adique de $p$ est donné par la formule :
$ deg_ {\ell}(p)
= \left\{
    \begin{array}{lll}
      Log_{Iw}(p) & \mbox{ si } p \ne  \ell \\
      Log_{Iw}(1+\ell) & \mbox{ si }  p=\ell 
    \end{array}
\right.$
\smallskip

\noindent ii) le degré $\ell$-adique de $\mathfrak{p}$ est défini par la formule :
$deg(\mathfrak{p})=\tilde{f_{\mathfrak{p}}}\cdot deg_{\ell}(p)$
\smallskip

\noindent iii) la valuation logarithmique associée à $\mathfrak{p}$ est : 
$ \tilde v_{\mathfrak{p}}(x)= -Log_{Iw}(N_{K_{\mathfrak{p}}/\mathbb{Q}_{p}}(x))/deg_{\ell}(\mathfrak{p}) $
, définie sur $\mathcal{R}_{K_{\mathfrak{p}}}$ et à valeurs dans $\mathbb{Z}_{\ell}$,~\cite{Ja2}[proposition 1.2]

\end{defi}
\smallskip

\noindent C'est le choix du dénominateur dans  l'expression de la valuation logarithmique qui impose le choix de l'uniformisante logarithmique $\tilde{\ell}$ sur $\mathbb{Q}_{\ell}$. Ainsi le choix de $\textrm{deg}_{\ell}(\ell)=\textrm{Log}_{Iw}(1+\ell)$ impose $\tilde{\ell}=1+\ell$, \cite{Re2}.
Pour $K_\mathfrak{p}$ une extension finie de $\mathbb{Q}_{p}$, nous avons :

\begin{defi}{Uniformisante logarithmique  de $\mathcal{R}_{K_{\mathfrak{p} }} $  \cite{Re2} }
\medskip

\noindent \textbf{Si} $\mathfrak{p} \not\vert \ell$, l'uniformisante classique $\pi_{\mathfrak{p}}$ est aussi 
uniformisante logarithmique.
\smallskip

\noindent \textbf{Si} $\mathfrak{p} \vert \ell$, une uniformisante logarithmique est un élément $\tilde{\pi}_{\mathfrak{p}}$
de $\mathcal{R}_{K_{\mathfrak{p}}}$ vérifiant :
$$ Log_{Iw}( N_{K_{\mathfrak{p}}/\mathbb{Q}_{\ell}}( \tilde \pi_{\mathfrak{p}})   )=\tilde{f}_{\mathfrak{p}} \; \textrm{deg}_{\ell}(\ell)=Log_{Iw}(\tilde{\ell}^{\tilde{f}_{\mathfrak{p}} })$$

\noindent où $\tilde{\ell}$ désigne l'uniformisante logarithmique de $\mathcal{R}_{\mathbb{Q}_{\ell}}$.

\end{defi}

\noindent Le noyau de cette valuation définit le groupe des unités logarithmiques $\widetilde{ \mathcal{U}}_{K_{\mathfrak{p}}}$ de
$\mathcal{R}_{K_{\mathfrak{p}}}$. Nous considérons une filtration décroissante du groupe des unités logarithmiques : $(\widetilde{ \mathcal{U}}_{K_{\mathfrak{p}}}^{n})_{n \in \mathbb{N}}$
à partir de laquelle nous posons les définitions suivantes :

\begin{defi}{Conducteurs logarithmiques \cite{Re2}}

\noindent{i) Si $L_{\mathfrak{P}}/K_{\mathfrak{p}}$ est une $\ell$-extension abélienne finie ,et si $n$ est le plus petit entier tel que }
$\widetilde{ \mathcal{U}}_{K_{\mathfrak{p}}}^{n}  \subseteq N_{L_{\mathfrak{P}}/K_{\mathfrak{p}}}( \mathcal{R}_{L_{\mathfrak{P}}})$ alors l'idéal :
$$ \tilde {\mathfrak{f}}_{\mathfrak{p}} = \mathfrak{p}^{n}$$
est appelé conducteur logarithmique local associé à l'extension.
\smallskip

\noindent{ii) Si $L/K$ est une $\ell$-extension abélienne finie, le conducteur logarithmique global est défini par  : }
$$\tilde{\mathfrak{f}}_{L/K}= \prod_{\mathfrak{p}}  \tilde{\mathfrak{f}}_{\mathfrak{p}}  $$

\end{defi}

\begin{defi}{L'application d'Artin logarithmique \cite{Re2}}

\noindent Soit $L/K$ une $\ell$-extension abélienne finie. Soit $\mathfrak{p}$ une place de $K$ logarithmiquement non ramifiée dans $L$.
\noindent Soit $D\ell_{K}$ le groupe des diviseurs logarithmiques de $K$. Soit $\tilde{\mathfrak{f}}_{L/K}$ le conducteur logarithmique global de de $L/K$ , et $D\ell_{K}^{\tilde{\mathfrak{f}}_{L/K}}$ le sous-groupe des diviseurs logarithmiques  premiers au conducteur $\tilde{\mathfrak{f}}_{L/K}$.

\noindent{Nous définissons le Frobenius logarithmique attaché à une place $\mathfrak{p}$ logarithmiquement non ramifiée : }

$$ (\widetilde{\frac{L/K}{\mathfrak{p}}})=([\tilde \pi_{\mathfrak{p}}], L/K)$$
avec $\tilde \pi_{\mathfrak{p}}$ l'uniformisante logarithmique définie préalablement et
$[\tilde \pi_{\mathfrak{p}}]$ l'image de $\tilde{\pi_{\mathfrak{p}}}$ dans $\mathcal{J}_K.$

\noindent{Nous obtenons ainsipar extension  l'application d'Artin logarithmique attachée à une place $\mathfrak{p}$ logarithmiquement non ramifiée: }

\begin{center}
$\begin{array}{ccccc}
\widetilde{(\frac{L/K}{\;})}   &:  D\ell_{K}^{\tilde{\mathfrak{f}}_{L/K}} &  & \to & \mathrm{Gal}(L/K) \\
& & \mathfrak{p} & \mapsto & (\widetilde{\frac{L/K}{\mathfrak{p}}})\\
\end{array}$
\end{center}

\end{defi}
\medskip

\begin{defi}{Le symbole de Hasse logarithmique}

\noindent Soient $L$ une $\ell$-extension abélienne finie de $K$, $\alpha$ un idèle principal $\alpha \in \mathcal{R}_{K}$ et $\mathfrak{p}$  une place de $K$.
Notons  $\tilde{\mathfrak{f}}_{\mathfrak{p}}$ le conducteur logarithmique local associé à $\mathfrak{p}$ et $\tilde{\mathfrak{f}}_{L/K}$ le conducteur logarithmique global de l'extension $L/K$.
Considérons  $\beta \in \mathcal{R}_{K}$ vérifiant 
$$ \frac{\beta}{\alpha} \in \mathcal{R}_{K}^{(\tilde{\mathfrak{f}}_{\mathfrak{p}})} \; \textrm{et} \; \beta \in \mathcal{R}_{K}^{(\frac{\tilde{\mathfrak{f}}_{L/K}} {\tilde{\mathfrak{f}}_{\mathfrak{p}}})} .$$
\noindent Nous dirons  que $\beta$ est un $\mathfrak{p}$-associé logarithmique de $\alpha$.  

\noindent Nous écrivons alors  $\psi(\beta)= \mathfrak{p}^{\mathfrak{a}} a$ avec $a$ premier à $\mathfrak{p}$ et $a \in D\ell_{K}^{\tilde{\mathfrak{f}}_{L/K}}$.

\noindent Nous définissons le symbole $( \widetilde{\frac{\alpha,L/K} { \mathfrak{p}})}$  de la façon suivante :
$$\widetilde{ ( \frac{\alpha,L/K} { \mathfrak{p}})}= (\widetilde{ \frac{L/K}{a}}). $$
\noindent Ce symbole, défini sur $\mathcal{R}_{K}$, est appelé le symbole de Hasse logarithmique pour $\mathfrak{p}$ dans $L/K$ \index{symbole de Hasse logarithmique}.
\noindent Dans le cas particulier où $\widetilde{\mathfrak{f}_{\mathfrak{p}}}=1$, nous remplaçons la condition $ \frac{\beta}{\alpha} \in \mathcal{R}_{K}^{(1)}$ par $ \frac{\beta}{\alpha} $ est premier à $\mathfrak{p}$.
\end{defi}
\medskip

\noindent \textbf{Remarques :}

\noindent i) L'existence d'un tel $\beta$ est assurée par le lemme suivant:
\begin{lemma}{Théorème d'approximation $\ell$-adique}{~\cite[II.2]{Gr}}

\noindent Le morphisme de semi-localisation
$$ \mathcal{R}_{K} \longrightarrow \prod_{\mathfrak{p} \in S} \mathcal{R}_{K_{\mathfrak{p}}}$$
\noindent est surjectif pour tout ensemble fini de places $S$. En effet, l'image de $\mathcal{R}_{K}$ est un sous-$\mathbb{Z}_{\ell}$-module compact et dense.

\end{lemma}

\noindent ii) Vérifions que le résultat ne dépend pas du choix du $\mathfrak{p}$-associé de $\alpha$. 

\noindent En effet, si $\beta^{'}$ est un autre $\mathfrak{p}$-associé de $\alpha$, alors nous avons $\frac{\beta}{\alpha} \in \mathcal{R}_{K}^{(\widetilde{\mathfrak{f}_{\mathfrak{p}}})} \; \textrm{et} \; \beta \in \mathcal{R}_{K}^{(\frac{\tilde{\mathfrak{f}}_{L/K}} {\tilde{\mathfrak{f}}_{\mathfrak{p}}})} $ et
$\frac{\beta^{'}}{\alpha} \in \mathcal{R}_{K}^{(\widetilde{\mathfrak{f}_{\mathfrak{p}}})} \; \textrm{et} \; \beta^{'} \in \mathcal{R}_{K}^{(\frac{\tilde{\mathfrak{f}}_{L/K}} {\tilde{\mathfrak{f}}_{\mathfrak{p}}})} $, à savoir 
$\frac{\beta^{'}}{\beta} \in \mathcal{R}_{K}^{(\widetilde{\mathfrak{f}_{\mathfrak{p}}})} \; \textrm{et} \; \frac{\beta^{'}} {\beta} \in \mathcal{R}_{K}^{(\frac{\widetilde{\mathfrak{f}_{L/K}}} {\tilde{\mathfrak{f}}_{\mathfrak{p}}})} $. Finalement,
nous avons donc $\frac{\beta^{'}}{\beta} \in \mathcal{R}_{K}^{ (\tilde{\mathfrak{f}}_{L/K})}$. \'Ecrivons $\psi( \beta^{'})= \mathfrak{p}^{\mathfrak{a'}} a'$, nous en déduisons alors que $\mathfrak{a}=\mathfrak{a'}$. Finalement, $\psi( \frac{a'}{a}) \in 
P\ell_{K}^{(\tilde{\mathfrak{f}}_{L/K})}$. Il suit que  $(\widetilde{ \frac{L/K}{a}})= (\widetilde{ \frac{L/K}{a'}})$ car
$ P\ell_{K}^{(\tilde{\mathfrak{f}}_{L/K})} \in A\ell_{L/K}$, noyau de l'application d'Artin logarithmique.
\bigskip

\begin{proposition}
Le symbole logarithmique de Hasse pour $\mathfrak{p}$ a les propriétés suivantes :

\noindent i) c'est un homomorphisme de $\mathcal{R}_{K}$ dans $\textrm{Gal}(L/K)$

\noindent ii) si $\mathfrak{p} \not \vert \widetilde{\mathfrak{f}_{L/K}}$, alors $ \widetilde{ ( \frac{\alpha,L/K} { \mathfrak{p}}) }= (\widetilde{ \frac{L/K}{\mathfrak{p}}})^{-\mathfrak{a}}$ où $\psi(\alpha)= \mathfrak{p}^{\mathfrak{a}} a $ avec $a$ premier à $\mathfrak{
p}$

\noindent iii) la restriction de $\widetilde{ ( \frac{\; .,L/K} { \mathfrak{p}}) }$ à $M \subset L$ est $\widetilde { ( \frac{\; .,M/K} { \mathfrak{p}}) }$.
\end{proposition}
\smallskip

\begin{proof}
\noindent i) $\beta$ et $\beta^{'}$ étant des $\mathfrak{p}$-associés de $\alpha$ et $\alpha^{'}$, alors $\beta \beta^{'}$ est un $\mathfrak{p}$-associé de $\alpha \alpha^{'}$. De l'écriture $\psi(\beta)=\mathfrak{p}^{\mathfrak{a}} a $ et $\psi(\beta^{'})=\mathfrak{p}^{\mathfrak{a'}} a' $, nous en déduisons que $\psi(\beta \beta^{'})= \mathfrak{p}^{\mathfrak{a}}
\mathfrak{p}^{\mathfrak{a'}} a  a'$. Il suit alors que  $\widetilde{ ( \frac{\alpha \alpha^{'},L/K} { \mathfrak{p}}) }=(\widetilde{ \frac{L/K}{a a'}})=(\widetilde{ \frac{L/K}{a}}) (\widetilde{ \frac{L/K}{a'}})= \widetilde{ ( \frac{\alpha,L/K} { \mathfrak{p} }) } \widetilde{  ( \frac{\alpha^{'},L/K} { \mathfrak{p} }) }$ car l'application d'Artin logarithmique est un homomorphisme.

\noindent ii) Comme par hypothèse $\mathfrak{p} \not \vert  \tilde{\mathfrak{f}}_{L/K}$, nous avons que
 $\tilde{\mathfrak{f}}_{\mathfrak{p}}=1$. Par définition $\beta$ un $\mathfrak{p}$-associé de $\alpha$ vérifie : $\frac{\beta}{\alpha}$ est premier à $\mathfrak{p}$ et $\beta \in  \mathcal{R}_{K}^{ (\tilde{\mathfrak{f}}_{L/K})}$. Ainsi 
$\psi(\beta)=\mathfrak{p}^{\mathfrak{a}} a  \in P\ell_{K}^{(\tilde{\mathfrak{f}}_{L/K})}$. Alors  $\widetilde{ ( \frac{\alpha,L/K} { \mathfrak{p}}) }=(\widetilde{ \frac{L/K}{a}})=\widetilde{ (\frac{L/K}{\psi(\beta) \mathfrak{p}^{- \mathfrak{a}}) }}=
(\widetilde{ \frac{L/K}{\psi(\beta)}}) (\widetilde{ \frac{L/K}{\mathfrak{p}^{-\mathfrak{a}} }) }=(\widetilde{ \frac{L/K}{\mathfrak{p}^{-\mathfrak{a}}}}) $ car l'autre terme est dans le noyau du morphisme d'Artin logarithmique.

\noindent iii) Soit $\beta$ un $\mathfrak{p}$-associé de $\alpha$ dans $L/K$, alors $\beta$ est encore
un $\mathfrak{p}$-associé de $\alpha$ dans $M/K$ pour $M \subset L$. En effet $\tilde{\mathfrak{f}}_{M/K} \vert
\tilde{\mathfrak{f}}_{L/K}$, de même $\tilde{\mathfrak{f}}_{ \mathfrak{p}, M/K }  \vert \tilde{\mathfrak{f}}_{\mathfrak{p}, L/K} $, et $\frac{\tilde{\mathfrak{f}}_{M/K}} { \tilde{\mathfrak{f}}_{ \mathfrak{p}, M/K } }  \vert \frac{\tilde{\mathfrak{f}}_{L/K} }  {\tilde{\mathfrak{f}}_{ \mathfrak{p}, L/K } }$.
Nous avons alors, si $\psi(\beta)=\mathfrak{p}^{\mathfrak{a}} a $,  $\widetilde{ ( \frac{\alpha,M/K} { \mathfrak{p}}) } =(\widetilde{ \frac{M/K}{a}})$. Or $(\widetilde{ \frac{M/K}{a}})$ est la restriction de $(\widetilde{ \frac{L/K}{a}})=\widetilde{( \frac{\alpha,L/K} { \mathfrak{p}}) }$ à $M$.
\end{proof}
\medskip

\begin{theorem}
L'image de $\mathcal{R}_{K}$ par le symbole de Hasse logarithmique pour $\mathfrak{p}$ dans $L/K$ est égale au sous-groupe de décomposition pour $\mathfrak{p}$ dans $L/K$.
\end{theorem}

\medskip

\begin{proof}
Nous rappelons ici que pour une place $\mathfrak{p}$ de $K$, être complètement décomposée au sens classique ou au sens logarithmique c'est la même chose. Désignons alors par $M$ le corps de décomposition de $\mathfrak{p}$, alors $\mathfrak{p}$ est complètement décomposée dans $M/K$ :  par la proposition précédente ii) et iii), nous en déduisons $\widetilde{ ( \frac{\alpha,L/K} { \mathfrak{p}}) }$ restreint à $M$ est $\widetilde{ ( \frac{\alpha,M/K} { \mathfrak{p}}) }$ et $\widetilde{ ( \frac{\alpha,M/K} { \mathfrak{p}}) } = (\widetilde{ \frac{M/K}{\mathfrak{p}^{-\mathfrak{a}}}})=
(\widetilde{ \frac{M/K}{\mathfrak{p}} })^{-\mathfrak{a}} $.  Or $(\widetilde{ \frac{M/K}{\mathfrak{p}} })^{-\mathfrak{a}} $ est le Frobenius logarithmique en $M$ : de part la définition de $M$, nous avons donc $(\widetilde{ \frac{M/K}{\mathfrak{p}} })^{-\mathfrak{a}}=1 $ et $\widetilde{ ( \frac{\alpha,L/K} { \mathfrak{p}}) } \in \textrm{Gal}(L/M)$ : donc l'image du symbole de Hasse logarithmique est contenue dans le sous groupe de décomposition.
\medskip

\noindent Réciproquement, il nous faut montrer que le symbole de Hasse logarithmique est surjectif sur le sous-groupe de  décomposition de $\mathfrak{p}$. Désignons alors par $I$ le corps d'inertie logarithmique de $\mathfrak{p}$, à savoir qui
correspond à l'extension maximale logarithmiquement non ramifiée :

\[
\xymatrix{
    L \ar@{-}[d]^{\widetilde{e_{\mathfrak{p}}} }  \\
    I \ar@{-}[d]^{ \widetilde{f_{\mathfrak{p}}}}  \\
    M \ar@{-}[d]   \\
    K
  }
\]

\noindent  Considérons $\sigma \in \textrm{Gal}(L/M) \subset \textrm{Gal}(L/K)$, alors par la surjectivité de l'application d'Artin logarithmique, nous pouvons trouver $ b \in D\ell_{K}^{\tilde{\mathfrak{f}}_{L/K}}$ tel que $(\widetilde{ \frac{L/K}{b} } )   =\sigma$. Le but est alors de montrer que 
  $(\widetilde{ \frac{L/K}{b} })$ se met sous la forme $\widetilde { ( \frac{\alpha,L/K} { \mathfrak{p}}) }$. Or $\sigma$ restreint à $M$ c'est l'identité, donc  $(\widetilde{ \frac{M/K}{b} } )=1$, et ainsi $(\widetilde{ \frac{I/K}{b} } ) \in \textrm{Gal}(I/M)$. Or nous savons que
$(\widetilde{ \frac{I/K}{\mathfrak{p}} } )$ est le Frobenius logarithmique en $\mathfrak{p}$  dans $I/K$, donc il est d'ordre $\tilde f_{\mathfrak{p}}$ et engendre le groupe de Galois $\textrm{Gal}(I/M)$. $(\widetilde{ \frac{I/K}{b} } )$ s'exprime donc comme une puissance du Frobenius logarithmique. Par suite il existe donc une puissance $\mathfrak{p}^{a}$ de $\mathfrak{p}$ telle que $\mathfrak{p}^{a} b \in A\ell_{I/K, {(\frac{\tilde{\mathfrak{f}}_{L/K}} {\tilde{\mathfrak{f}}_{\mathfrak{p}}}) }              }$. L'égalité du lemme suivant $A\ell_{K} P\ell_{K}^{{(\frac{\widetilde{\mathfrak{f}_{L/K}}} {\widetilde{\mathfrak{f}_{\mathfrak{p}}}}) }}= A\ell_{I/K, {(\frac{\tilde{\mathfrak{f}}_{L/K}} {\tilde{\mathfrak{f}}_{\mathfrak{p}}}) }              }$ permet alors d'écrire :
$\mathfrak{p}^{a} b=\psi(\alpha) c$ avec $\alpha \in \mathcal{R}_{K}^{(\frac{\tilde{\mathfrak{f}}_{L/K}} {\tilde{\mathfrak{f}}_{\mathfrak{p}}}) }$ et $c \in A\ell_{K}$. Nous considérons alors  $( \frac{\alpha,L/K} { \mathfrak{p}})$ :
$\alpha$ est un $\mathfrak{p}$-associé de lui même, et nous avons $ \psi(\alpha)=\mathfrak{p}^{a} b c^{-1}$.
Ainsi  $\widetilde{ ( \frac{\alpha,L/K} { \mathfrak{p}}) }=(\widetilde{ \frac{L/K}{b c^{-1}} })=(\widetilde{ \frac{L/K}{b} })$.
 \end{proof}
\bigskip

\begin{lemma}
Si $I$ est le corps d'inertie logarithmique dans $L/K$, pour la place $\mathfrak{p}$ de $K$, alors nous avons l'égalité suivante :
$$A\ell_{L/K} P\ell_{K}^{{(\frac{\tilde{\mathfrak{f}}_{L/K}} {\tilde{\mathfrak{f}}_{\mathfrak{p}}}) }}= A\ell_{I/K, {(\frac{\tilde{\mathfrak{f}}_{L/K}} {\tilde{\mathfrak{f}}_{\mathfrak{p}}}) }              }$$

\end{lemma}
\medskip

\begin{proof}
Nous avons d'abord l'inclusion $A\ell_{L/K} P\ell_{K}^{{(\frac{\tilde{\mathfrak{f}}_{L/K}} {\tilde{\mathfrak{f}}_{\mathfrak{p}}}}) } \subset A\ell_{I/K, {(\frac{\tilde{\mathfrak{f}}_{L/K}} {\tilde{\mathfrak{f}}_{\mathfrak{p}}}})               }$ : en effet, $I \subset L$ donc $\tilde{\mathfrak{f}}_{I/K} \vert
\tilde{\mathfrak{f}}_{L/K}$ et $\tilde{\mathfrak{f}}_{\mathfrak{p},I/K} \vert
\tilde{\mathfrak{f}}_{\mathfrak{p},L/K}$. Or $\mathfrak{p}$ est une place logarithmiquement non ramifiée dans $I$, donc
$\tilde{\mathfrak{f}}_{\mathfrak{p},I/K} =1$.  Il suit que $\tilde{\mathfrak{f}}_{I/K} \vert \frac{\tilde{\mathfrak{f}}_{L/K}} {\tilde{\mathfrak{f}}_{\mathfrak{p}}}$, nous en déduisons donc que 
$P\ell_{K}^{(\frac{\tilde{\mathfrak{f}}_{L/K}} {\tilde{\mathfrak{f}}_{\mathfrak{p}}})} \subset P\ell_{K}^{ (\tilde{\mathfrak{f}}_{I/K})} \subset A\ell_{I/K}$ de part l'égalité des groupes d'Artin et de Takagi.
D'autre part $\frac{\tilde{\mathfrak{f}}_{L/K}} {\tilde{\mathfrak{f}}_{\mathfrak{p}}} \vert \tilde{\mathfrak{f}}_{L/K}$, nous avons donc $D\ell_{K}^{\tilde{\mathfrak{f}}_{L/K}}  \subset D\ell_{K}^{\frac{\tilde{\mathfrak{f}}_{L/K}} {\tilde{\mathfrak{f}}_{\mathfrak{p}}}} $ et par définition du noyau d'Artin,
nous obtenons $A\ell_{L/K} \subset D\ell_{K}^{\tilde{\mathfrak{f}}_{L/K}}  \subset D\ell_{K}^{\frac{\tilde{\mathfrak{f}}_{L/K}} {\tilde{\mathfrak{f}}_{\mathfrak{p}}} }$. Et comme $I \subset L$, 
nous avons $A\ell_{L/K} \subset A\ell_{I/K}$, d'où la première inclusion.

\noindent Considérons alors les sous-modules de congruences ~\cite[p. 70]{Re3} $(\frac{\tilde{\mathfrak{f}}_{L/K}} {\tilde{\mathfrak{f}}_{\mathfrak{p}}} ,  A\ell_{L/K} P\ell_{K}^{{(\frac{\tilde{\mathfrak{f}}_{L/K}} {\tilde{\mathfrak{f}}_{\mathfrak{p}}}) }} )$ et $(\frac{\tilde{\mathfrak{f}}_{L/K}} {\tilde{\mathfrak{f}}_{\mathfrak{p}}},   A\ell_{I/K, {(\frac{\tilde{\mathfrak{f}}_{L/K}} {\tilde{\mathfrak{f}}_{\mathfrak{p}}}) }  } )$.
Ainsi nous déduisons de l'inclusion  $A\ell_{L/K} P\ell_{K}^{{(\frac{\tilde{\mathfrak{f}}_{L/K}} {\tilde{\mathfrak{f}}_{\mathfrak{p}}}}) } \subset A\ell_{I/K, {(\frac{\tilde{\mathfrak{f}}_{L/K}} {\tilde{\mathfrak{f}}_{\mathfrak{p}}}) }              }$, et de la correspondance
du corps des classes ~\cite[p .76]{Re3} : $A\ell_{I, {(\frac{\tilde{\mathfrak{f}}_{L/K}} {\tilde{\mathfrak{f}}_{\mathfrak{p}}}) }              }$ est associé à $I$, et $A\ell_{K} P\ell_{K}^{{(\frac{\tilde{\mathfrak{f}}_{L/K}} {\tilde{\mathfrak{f}}_{\mathfrak{p}}}) }}$ est associé à la sous-extension $I'$ intermédiaire entre $I$ et $K$. Or $\mathfrak{p}$ ne divise pas $(\frac{\tilde{\mathfrak{f}}_{L/K}} {\tilde{\mathfrak{f}}_{\mathfrak{p}}})$, et le conducteur de la classe de congruences
a les mêmes diviseurs premiers que le conducteur logarithmique global ~\cite[p .78]{Re3}: par suite nous en déduisons que 
 $\mathfrak{p}$ est une place logarithmiquement non ramifiée dans $I'$.
Et de ce fait $I'=I$, nous obtenons alors l'équivalence des deux sous-modules de congruences ci dessus, et
finalement l'égalité.
\end{proof}
\medskip

\begin{proposition}{Caractérisation du noyau du symbole de Hasse logarithmique}

\noindent Une condition nécessaire et suffisante pour que $\widetilde{ ( \frac{\alpha,L/K} { \mathfrak{p}}) }=1$, où $\alpha \in \mathcal{R}_{K}$ est que $\alpha$ soit norme locale en $\mathfrak{p}$.
\end{proposition}
\medskip

\begin{proof}
Supposons que $\alpha$ soit norme locale en $\mathfrak{p}$ alors $\alpha_{\mathfrak{p}} \in N_{L_{\mathfrak{P}}/K_{\mathfrak{p}} } \mathcal{R}_{L_{\mathfrak{P}}}$. Il existe donc $z_{\mathfrak{P}} \in \mathcal{R}_{L_{\mathfrak{P}}}$ tel que $\alpha_{\mathfrak{p}} = N_{L_{\mathfrak{P}}/K_{\mathfrak{p}} } (z_{\mathfrak{P}})$.
\'Ecrivons alors $\alpha =N_{L/K}(x) u$ avec $ x \in \mathcal{R}_{L}$ dont toutes les coordonnées valent $1$, sauf $x_{\mathfrak{P}}$ pour $\mathfrak{P} \vert \mathfrak{p}$ que nous prenons égal à $z_{\mathfrak{P}}$ ; et $u \in \mathcal{R}_{K}^{(\tilde{ \mathfrak{f}}_{\mathfrak{p}} )}$ (en fait $u=\alpha$ en dehors de la $\mathfrak{p}$-ième
composante que l'on prend égale à $1$). Alors grâce au lemme d'approximation $\ell$-adique, nous pouvons trouver $y \in \mathcal{R}_{L}$ 
tel que $\frac{y}{x} \in \mathcal{R}_{L}^{(\tilde{ \mathfrak{f}}_{\mathfrak{p}} )}$ et $y \in \mathcal{R}_{L}^{(\frac{ \tilde{\mathfrak{f}}_{L/K}   } {\tilde{ \mathfrak{f}}_{\mathfrak{p}} })}$, les modules étant les étendus à $L$.
Ainsi nous avons $N_{L/K}(\frac{y}{x}) \in \mathcal{R}_{K}^{(\tilde{ \mathfrak{f}}_{\mathfrak{p} })}$ et $N_{L/K}(y) \in \mathcal{R}_{K}^{(\frac{ \tilde{\mathfrak{f}}_{L/K}} {\tilde{ \mathfrak{f}}_{\mathfrak{p} }})}$.
Par suite, $N_{L/K}(y)$ est un $\mathfrak{p}$-associé de $\alpha$. \'Ecrivons $\psi(y)= \mathcal{O} \mathcal{U}$ avec $\mathcal{O}$ qui contient les diviseurs premiers au dessus de $\mathfrak{p}$ et $\mathcal{U}$ étant premier à $\mathfrak{p}$ i.e $\mathcal{U} \in P\ell_{L}^{(\tilde{\mathfrak{f}}_{L/K})}$. Ainsi $\psi(N_{L/K}(y))= N_{L/K}(\psi(y))=N_{L/K}(\mathcal{O}) N_{L/K}(\mathcal{U})$. Il suit alors $\widetilde{ ( \frac{\alpha,L/K} { \mathfrak{p}}) }=(\widetilde{\frac{L/K}{N_{L/K} (\mathcal{U}) }})=1$ puisque 
$N_{L/K} P\ell_{L}^{(\tilde{\mathfrak{f}}_{L/K})} \subset A\ell_{K}$.
\medskip

\noindent Réciproquement supposons $\widetilde{ ( \frac{\alpha,L/K} { \mathfrak{p}}) }=1$, et désignons par $X_{\mathfrak{p}}$  l'ensemble
des éléments $\alpha \in \mathcal{R}_{K}$ tels que $\widetilde{ ( \frac{\alpha,L/K} { \mathfrak{p}}) }=1$. Alors nous avons la suite exacte suivante :
$$ 1 \longrightarrow X_{\mathfrak{p}} \longrightarrow \mathcal{R}_{K} \longrightarrow D_{\mathfrak{p}} \longrightarrow 1$$
\noindent où $D_{\mathfrak{p}}$ désigne le sous-groupe de décomposition de $\mathfrak{p}$.
D'après la partie directe, nous savons que $\mathcal{R}_{K} \cap N_{L_{\mathfrak{P}}/K_{\mathfrak{p}} } \mathcal{R}_{L_{\mathfrak{P}}} \subseteq  X_{\mathfrak{p}} $. Or $\frac{\mathcal{R}_{K_{\mathfrak{p}}} } {N_{L_{\mathfrak{P}}/K_{\mathfrak{p}} } \mathcal{R}_{L_{\mathfrak{P}}}}$ est isomorphe au groupe de Galois local, et est d'indice $e_{\mathfrak{p}} f_{\mathfrak{p}}= \widetilde{e_{\mathfrak{p}}} \widetilde{f_{\mathfrak{p}}}$.
Or $\frac {\mathcal{R}_{K_{\mathfrak{p}}}  } { N_{L_{\mathfrak{P}}/K_{\mathfrak{p}} } \mathcal{R}_{L_{\mathfrak{P}}} } \simeq
\frac{\mathcal{R}_{K}} {\mathcal{R}_{K} \cap N_{L_{\mathfrak{P}}/K_{\mathfrak{p}} } \mathcal{R}_{L_{\mathfrak{P}}} }$,
par suite nous en déduisons l'égalité $X_{\mathfrak{p}} =\mathcal{R}_{K} \cap N_{L_{\mathfrak{P}}/K_{\mathfrak{p}} } \mathcal{R}_{L_{\mathfrak{P}}}  $.
\end{proof}
\medskip

\medskip

\noindent \textbf{Remarque :}
Le symbole de Hasse logarithmique permet également, comme dans le cas classique, une description du sous-groupe d'inertie logarithmique
associé à une place $\mathfrak{p}$ et une $\ell$-extension abélienne finie $L/K$, celui étant défini comme le sous-groupe du groupe de Galois de l'extension, qui fixe l'extension maximale de $K$ logarithmiquement non ramifiée.
\medskip

\begin{theorem}
$L/K$ étant une $\ell$-extension abélienne finie, la restriction du symbole de Hasse logarithmique aux idèles principaux dont le diviseur
logarithmique principal est premier à $\mathfrak{p}$ a pour image le sous-groupe d'inertie logarithmique pour $\mathfrak{p}$ dans $L/K$.
\index{corps d'inertie logarithmique}
\end{theorem}
\medskip

\begin{proof}
Notons $\mathcal{R}_{K,\mathfrak{p}}$ l'ensemble des idèles principaux dont le diviseur logarithmique principal est premier à $\mathfrak{p}$.
Si $\alpha \in \mathcal{R}_{K,\mathfrak{p}}$, un $\mathfrak{p}$-associé de $\alpha$, $\beta$ a pour image un diviseur logarithmique principal $\psi(\beta)$ premier à $\mathfrak{p}$ puisque $\frac{\beta}{\alpha} \in \mathcal{R}_{K}^{(\tilde{\mathfrak{f}}_{\mathfrak{p}})}$, et
$\widetilde{ ( \frac{\alpha,L/K} { \mathfrak{p}})}=(\widetilde{\frac{L/K}{\psi(\beta)}})$. Désignons par $I$ le corps d'inertie logarithmique
et par $I_{\mathfrak{p}}$ le sous-groupe d'inertie logarithmique. Par définition du $\mathfrak{p}$-associé, nous savons que $\beta \in \mathcal{R}_{K}^{(\frac{\tilde{\mathfrak{f}}_{L/K}} {\tilde{\mathfrak{f}}_{\mathfrak{p}}}) }$, Or $\tilde{\mathfrak{f}}_{I/K} \vert 
\tilde{\mathfrak{f}}_{L/K}=\tilde{\mathfrak{f}}_{\mathfrak{p}} \frac{\tilde{\mathfrak{f}}_{L/K}} {\tilde{\mathfrak{f}}_{\mathfrak{p}}}$, comme $\tilde{\mathfrak{f}}_{I/K}$ est premier avec $\tilde{\mathfrak{f}}_{\mathfrak{p}}$, nous obtenons que $\frac{\tilde{\mathfrak{f}}_{L/K}} {\tilde{\mathfrak{f}}_{\mathfrak{p}}}$ est un multiple de $\tilde{\mathfrak{f}}_{I/K}$, ainsi de part l'expression du noyau de l'application d'Artin : $(\widetilde{\frac{I/K}{\psi(\beta)}})=1$. 
\smallskip

\noindent Réciproquement, supposons que $\sigma \in I_{\mathfrak{p}}$, par la surjectivité du symbole d'Artin logarithmique il existe $b \in D\ell_{K}^{\tilde{\mathfrak{f}}_{L/K}}$ tel que $(\widetilde{\frac{L/K}{b}})=\sigma.$
Comme $\sigma \in I_{\mathfrak{p}} $, 
nous avons $(\widetilde{\frac{I/K}{b}})=1$. Finalement, $b \in A\ell_{I/K,\tilde{\mathfrak{f}}_{L/K}} \subset
A\ell_{I/K,\frac{\tilde{\mathfrak{f}}_{L/K}} {\tilde{\mathfrak{f}}_{\mathfrak{p}}}}$. 
Par le lemme précédent, nous savons que 
$A\ell_{L/K} P\ell_{K}^{(\frac{\tilde{\mathfrak{f}}_{L/K}} {\tilde{\mathfrak{f}}_{\mathfrak{p}}}) }= A\ell_{I/K, {(\frac{\tilde{\mathfrak{f}}_{L/K}} {\tilde{\mathfrak{f}}_{\mathfrak{p}}}) }             }$ : il existe donc  $\alpha \in \mathcal{R}_{K}^{(\frac{\tilde{\mathfrak{f}}_{L/K}} {\tilde{\mathfrak{f}}_{\mathfrak{p}}}) }$ et $c \in A\ell_{L/K}$ tels que $b=c \psi(\alpha)$ et $\alpha $ est premier avec $\mathfrak{p}$ car $b$ et $c$ le sont. Ainsi, comme $\alpha$ est son propre $\mathfrak{p}$-associé, nous obtenons :
$\widetilde{ ( \frac{\alpha,L/K} { \mathfrak{p}})}=(\widetilde{\frac{L/K}{\psi(\alpha)}})=(\widetilde{\frac{L/K}{b c^{-1}}})=(\widetilde{\frac{L/K}{b}})=\sigma.$

\end{proof}

\subsection{La formule du produit et sa réciproque pour le symbole de Hasse logarithmique}
\medskip

\begin{theorem}{Formule du produit pour le symbole de Hasse logarithmique}

\noindent Soit $L/K$ une $\ell$-extension abélienne finie et $\alpha \in \mathcal{R}_{K}$, alors
le symbole de Hasse logarithmique vérifie la formule du produit :

$$ \prod_{\mathfrak{p}} \widetilde{ ( \frac{\alpha,L/K} { \mathfrak{p}} )}=1 $$

\end{theorem}

\begin{proof}
\noindent Ce produit a tout d'abord un sens en vertu de la proposition 2.1.1 précédente :
 si $\mathfrak{p} \not \vert \widetilde{\mathfrak{f}_{L/K}}$, alors $ \widetilde{ ( \frac{\alpha,L/K} { \mathfrak{p}}) }= (\widetilde{ \frac{L/K}{\mathfrak{p}}})^{-\mathfrak{a}}$ où $\psi(\alpha)= \mathfrak{p}^{\mathfrak{a}} a $ avec $a$ premier à $\mathfrak{
p}$. 
 
\noindent Il suit donc que si $\mathfrak{p} \not\vert  \widetilde{\mathfrak{f}_{L/K}} $ et $\mathfrak{p} \not \vert \psi(\alpha)$ (i.e $\mathfrak{a}=0$) alors $ \widetilde{ ( \frac{\alpha,L/K} { \mathfrak{p}})} =1$.
Ainsi, le symbole de Hasse logarithmique ne peut être différent de $1$, que pour les places logarithmiquement ramifiées dans $L/K$, ( donc divisant le conducteur logarithmique global $\widetilde{\mathfrak{f}_{L/K}}$ ) et les places divisant $\psi(\alpha).$ 

\noindent Soit $\alpha \in \mathcal{R}_{K}$, posons alors :
$$ \psi(\alpha)= \prod_{i=1}^{i=\rho} \mathfrak{p}_{i}^{\lambda_i} $$
\noindent avec dans $ \{ \mathfrak{p}_{1}, \mathfrak{p}_{2}, ..., \mathfrak{p}_{\rho} \}$  outre les diviseurs premiers
de $\psi(\alpha)$, les places finies logarithmiquement ramifiées. Enfin, notons  $\mathfrak{p}_{\rho+1}, ..., \mathfrak{p}_{R} $
les places à l'infini logarithmiquement ramifiées qui par convention sont les places à l'infini ramifiées au sens classique \cite{Ja2}. 

\noindent Pour $i=1...R$ considérons  $\beta_{i}$ un $\mathfrak{p}_{i}$-associé de $\alpha$ alors par définition :
$$ \frac{\beta_{i}}{\alpha} \in \mathcal{R}_{K}^{(\tilde{\mathfrak{f}}_{\mathfrak{p}_{i}})} \; \textrm{et} \; \beta_{i} \in \mathcal{R}_{K}^{(\frac{\tilde{\mathfrak{f}}_{L/K}} {\tilde{\mathfrak{f}}_{\mathfrak{p}_{i}}})} .$$
Nous posons $\psi(\beta_{i})=\mathfrak{p}_{i}^{\mathfrak{a}_{i}} a_{i}$ et  ainsi $ \widetilde{ ( \frac{\alpha,L/K} { \mathfrak{p}} )}=(\frac{L/K} {a_{i}})$, avec la convention $\mathfrak{a}_{i}=0$ pour les places à l'infini. Par hypothèse, 
$ \frac{\beta_{i}}{\alpha} \in \mathcal{R}_{K}^{(\tilde{\mathfrak{f}}_{\mathfrak{p}_{i}})}$, il en découle que 
$\psi(\frac{\beta_{i}}{\alpha})$ est premier à $\mathfrak{p}_{i}$. Or nous avons :
$$ \psi(\frac{\beta_{i}}{\alpha})= \frac {\mathfrak{p}_{i}^{\mathfrak{a}_{i}} a_i} {\prod_{j=1}^{\rho} \mathfrak{p}_{j}^{\lambda_{j}} }= \frac {\mathfrak{p}_{i}^{\mathfrak{a}_{i}} a_i} {\mathfrak{p}_{i}^{\lambda_{i}}
\prod_{j \ne i} \mathfrak{p}_{j}^{\lambda_{j}} }.$$
Nous en déduisons donc que pour tout $i \in [1,R] \;, \lambda_{i}=\mathfrak{a}_{i}$, et finalement nous obtenons :
$$ \psi(\alpha)= \prod_{i=1}^{\rho} \mathfrak{p}_{i}^{\mathfrak{a}_{i}}.$$
\noindent Il s'ensuit alors :
$$\prod_{i=1}^{R} \widetilde{ ( \frac{\alpha,L/K} { \mathfrak{p}_{i}} )}    = \prod_{i=1}^{R} (\widetilde{ \frac{L/K}{a_{i}}})=
(\widetilde{ \frac{L/K}{\prod_{i=1}^{R} a_{i}}})$$
\noindent d'après la propriété de multiplicativité du symbole d'Artin. 

\noindent Le but est de prouver que $\prod_{i=1}^{R} a_{i}$  est dans le noyau de l'application d'Artin logarithmique.

\noindent Or nous avons $ \prod_{i=1}^{R} \psi(\beta_{i})= \prod_{i=1}^{R} \mathfrak{p}_{i}^{\mathfrak{a}_{i}} \prod_{i=1}^{R} a_{i}.$ Comme $\mathfrak{a}_{i}=0$ pour les places à l'infini, il vient donc :
$$  \prod_{i=1}^{R} \psi(\beta_{i})= \psi(\alpha) \prod_{i=1}^{R} a_{i} \qquad \textrm{i.e} \qquad
\prod_{i=1}^{R} a_{i}= \frac{ \prod_{i=1}^{R} \psi(\beta_{i})} {\psi(\alpha)}= \psi(\frac{\prod_{i=1}^{R} \beta_{i}} {\alpha}) . $$

\noindent Soit alors $j \in \{1,..., R \}$, nous avons :
$$ \prod_{i \ne j}  \frac{\beta_{i}}{\alpha} \in \mathcal{R}_{K}^{(\tilde{\mathfrak{f}}_{\mathfrak{p}_{j}})} 
\; \textrm{et} \;   \frac{\beta_{j}}{\alpha} \in \mathcal{R}_{K}^{(\tilde{\mathfrak{f}}_{\mathfrak{p}_{j}})}. $$
Nous en déduisons donc :
$$ \prod_{i=1}^{R}   \frac{\beta_{i}}{\alpha}  \in \mathcal{R}_{K}^{(\tilde{\mathfrak{f}}_{\mathfrak{p}_{j}})} .$$
\noindent Les modules $\tilde{\mathfrak{f}}_{\mathfrak{p}_{j}}$ étant deux à deux premiers entre eux (chacun est une puissance de $\mathfrak{p}_{j}$), il vient :
$$  \prod_{i=1}^{R}   \frac{\beta_{i}}{\alpha}  \in \mathcal{R}_{K}^{(\tilde{\mathfrak{f}}_{L/K})} \Rightarrow 
\psi(\prod_{i=1}^{R}   \frac{\beta_{i}}{\alpha}) \in P\ell_{L/K}^{(\tilde{\mathfrak{f}}_{L/K})}.$$
Finalement, nous obtenons, compte-tenu de l'expression du sous-module d'Artin logarithmique :
$$   \prod_{i=1}^{R} \widetilde{ ( \frac{\alpha,L/K} { \mathfrak{p_{i}}} )}=1 $$
\noindent d'où le résultat puisqu'en dehors de ces places ce produit vaut déjà $1$.

\end{proof}

\begin{theorem}{Réciproque de la formule du produit}

\noindent Soit $L/K$ une $\ell$-extension abélienne finie, $S$ un ensemble fini de places de $K$, $(\sigma_{\mathfrak{p}})_{\mathfrak{p}} \in
\prod_{\mathfrak{p}} \mathcal{D}_{\mathfrak{p}}$, où $  \mathcal{D}_{\mathfrak{p}} $ désigne le sous-groupe de décomposition de $\mathfrak{p}$ dans $L/K$, tel que :

$$ \sigma_{\mathfrak{p}} =1  \textrm{ pour }  \mathfrak{p} \not \in S \; \textrm{et} \;
\prod_{\mathfrak{p}} \sigma_{\mathfrak{p}} =1  $$

\noindent alors il existe $\alpha \in \mathcal{R}_{K}$ vérifiant :
$$  \widetilde{ ( \frac{\alpha,L/K} { \mathfrak{p}} )}=\sigma_{\mathfrak{p}} $$
\noindent pour toute place $\mathfrak{p}$.

\end{theorem}

\begin{proof}

\noindent Tout comme dans le cas classique, nous ne changeons pas le problème en supposant que toute place
logarithmiquement ramifiée est dans $S$, à condition de supposer  : $\sigma_{\mathfrak{p}}=1$ pour toute place
$\mathfrak{p}$ logarithmiquement ramifiée rajoutée. 

\noindent Notons alors $\tilde{\mathfrak{f}}_{L/K}$
le conducteur logarithmique global de l'extension et considérons $\mathfrak{m}$ un multiple de ce dernier, tel que toute place de $S$ figure dans $\mathfrak{m}$ et seulement
les places de $S$. Comme l'image du symbole de Hasse logarithmique est le sous-groupe de décomposition de la place
$\mathfrak{p}$, nous en déduisons :
$$ \forall \mathfrak{p} \in S \; \exists \alpha_{\mathfrak{p}} \in \mathcal{R}_{K} \; \textrm{tel que} \; \widetilde{ ( \frac{\alpha_{\mathfrak{p}},L/K} { \mathfrak{p}} )}=\sigma_{\mathfrak{p}}. $$

\noindent Désignons par $\mathfrak{m}_{\mathfrak{p}}$ la $\mathfrak{p}$-participation de $\mathfrak{p}$ dans $\mathfrak{m}$ alors nous avons :

$\tilde{\mathfrak{f}}_{L/K} \vert \mathfrak{m} \Rightarrow \mathcal{R}_{K}^{(\mathfrak{m})} \subset \mathcal{R}_{K}^{(\tilde{\mathfrak{f}}_{L/K})} \; \textrm{et} \; \tilde{\mathfrak{f}}_{\mathfrak{p}} \vert \mathfrak{m}_{\mathfrak{p}} \Rightarrow\frac{ \tilde{\mathfrak{f}_{L/K}} } { \tilde{\mathfrak{f}}_{\mathfrak{p}} }  \vert \frac{\mathfrak{m}} {\mathfrak{m}_{\mathfrak{p}}} \quad
\; \textrm{il s'ensuit } \; \mathcal{R}_{K}^{(\frac{\mathfrak{m}}{\mathfrak{m}_{\mathfrak{p}}} )} \subset \mathcal{R}_{K}^{( \frac{ \tilde{\mathfrak{f}_{L/K}} } { \tilde{\mathfrak{f}}_{\mathfrak{p}} } )}.$

\noindent Pour chaque $\mathfrak{p} \in S$, nous pouvons donc trouver un $\mathfrak{p}$-associé $\beta_{\mathfrak{p}}$
de $\alpha_{\mathfrak{p}}$ vérifiant : (les conditions imposées au niveau du $\mathfrak{p}$-associé sont un peu plus restrictives
que d'habitude)

$$ \frac{\beta_{\mathfrak{p}}} {\alpha_{\mathfrak{p}} }  \in \mathcal{R}_{K}^{(\mathfrak{m}_{\mathfrak{p}})} \subset \mathcal{R}_{K}^{(\tilde{f}_{\mathfrak{p}})}$$ et $$ \beta_{\mathfrak{p}} \in \mathcal{R}_{K}^{(\frac{\mathfrak{m}}{\mathfrak{m}_{\mathfrak{p}}} )}.$$

\noindent Nous posons alors $\psi(\beta_{\mathfrak{p}})= \mathfrak{p}^{\mathfrak{a}_{\mathfrak{p}}} b_{\mathfrak{p}}.$
Comme $\beta_{\mathfrak{p}} \in \mathcal{R}_{K}^{(\frac{\mathfrak{m}}{\mathfrak{m}_{\mathfrak{p}}})}$, nous en déduisons que $\psi(\beta_{\mathfrak{p}})$ est premier avec $\frac{\mathfrak{m}}{\mathfrak{m}_{\mathfrak{p}}}$.
Or $b_{\mathfrak{p}}$ est choisi premier à $\mathfrak{p}$, donc $b_{\mathfrak{p}}$ est premier à $\mathfrak{m}$ (
sinon $\mathfrak{m} \vert b_{\mathfrak{p}} $ et comme $b_{\mathfrak{p}} $ est premier à $\mathfrak{p}$, nous aurions
$\frac{\mathfrak{m}}{\mathfrak{m}_{\mathfrak{p}} } \vert \psi(\beta_{\mathfrak{p}})$. ) Il suit $b_{\mathfrak{p}} \in D\ell_{K}^{\mathfrak{m}}.$

\noindent Par hypothèse, nous avons $\prod_{\mathfrak{p} \in S} \sigma_{\mathfrak{p}} =1  $, d'où
$$    \prod_{\mathfrak{p} \in S} (\widetilde{ \frac{L/K}{b_{\mathfrak{p}}} })=1  \; \textrm{i.e} \;
(\widetilde{ \frac{L/K}{\prod_{\mathfrak{p} \in S} b_{\mathfrak{p}}} })=1 .$$

\noindent Il s'ensuit alors $\prod_{\mathfrak{p} \in S} b_{\mathfrak{p}} \in T\ell_{L/K}$ et comme
$\prod_{\mathfrak{p} \in S} b_{\mathfrak{p}} \in D\ell_{K}^{\mathfrak{m}}$ ; nous en déduisons finalement :
$\prod_{\mathfrak{p} \in S} b_{\mathfrak{p}}  \in A\ell_{L/K,\mathfrak{m}}$. Par suite, il existe $u \in \mathcal{R}_{K}^{(\mathfrak{m})}$ et
$\mathcal{U} \in D\ell_{L}^{\mathfrak{m}}$ tels que 
$$   \prod_{\mathfrak{p} \in S} b_{\mathfrak{p}} = \psi(u) N_{L/K}(\mathcal{U}) .$$

\noindent Posons alors $\alpha= u^{-1} \prod_{\mathfrak{p} \in S} \beta_{\mathfrak{p}} $, ainsi il vient :

$$ \psi(\alpha)= \psi(u^{-1}) \prod_{\mathfrak{p} \in S} \psi(\beta_{\mathfrak{p}}) =
\psi(u^{-1}) \prod_{\mathfrak{p} \in S} \mathfrak{p}^{\mathfrak{a}_{\mathfrak{p}} } \prod_{\mathfrak{p} \in S} b_{\mathfrak{p}} =
 \prod_{\mathfrak{p} \in S} \mathfrak{p}^{\mathfrak{a}_{\mathfrak{p}} } \;  N_{L/K}(\mathcal{U}).$$
\smallskip

\noindent Nous cherchons alors à exprimer $\widetilde{(\frac{\alpha, L/K}{\mathfrak{q}} })$ pour toute place $\mathfrak{q}$ de $K$.
\bigskip

\noindent -Si $\mathfrak{q} \in S$, $\beta_{\mathfrak{q}}$ est alors un $\mathfrak{q}$-associé de $\alpha$.

\noindent En effet, par construction nous avons d'une part :  $\beta_{\mathfrak{q}} \in \mathcal{R}_{K}^{( \frac{ \tilde{\mathfrak{f}_{L/K}} } { \tilde{\mathfrak{f}}_{\mathfrak{q}} } )}$. Et d'autre part, nous obtenons 
$ \frac{\alpha}{\beta_{\mathfrak{q}} } = u^{-1} \prod_{\mathfrak{p} \in S \; \mathfrak{p} \ne \mathfrak{q} } \beta_{\mathfrak{p}}. $
\noindent Or nous savons  $  \beta_{\mathfrak{p}} \in \mathcal{R}_{K}^{(\frac{\mathfrak{m}}{\mathfrak{m}_{\mathfrak{p}}})}$ pour chaque $\mathfrak{p}$ donc $\prod_{\mathfrak{p} \in S \; \mathfrak{p} \ne \mathfrak{q} } \beta_{\mathfrak{p}} \in \mathcal{R}_{K}^{(\mathfrak{m}_{\mathfrak{q}})}$
et $u \in \mathcal{R}_{K}^{(\mathfrak{m})} \subset \mathcal{R}_{K}^{(\mathfrak{m}_{\mathfrak{q}})} \; \textrm{car} \; \mathfrak{m}_{\mathfrak{q}} \vert \mathfrak{m}.$ 
Ainsi il découle $\frac{\alpha}{\beta_{\mathfrak{q}} }  \in \mathcal{R}_{K}^{(\tilde{\mathfrak{f}}_{\mathfrak{q}})}$. 

\noindent Finalement, nous obtenons $ \widetilde{(\frac{\alpha, L/K} {\mathfrak{q}} )}=(\widetilde{\frac{L/K}{b_{\mathfrak{q}}} } )= \sigma_{\mathfrak{q}}$.

\medskip

\noindent -Si $\mathfrak{q} \not \in S$, alors nous avons fait en sorte que $\mathfrak{q}$ soit logarithmiquement non ramifiée 
i.e $\mathfrak{q} \not \vert \tilde{\mathfrak{f}}_{L/K}$. D'après la proposition 2.1.1 ii), le symbole de Hasse ne dépend alors que de la $\mathfrak{q}$-valuation de $\psi(\alpha)$. Or celle-ci ne dépend que de la $\mathfrak{q}$-valuation de $N_{L/K}(\mathcal{U})$ car $\mathfrak{q} \not \in S$ : elle est donc divisible par le degré d'inertie logarithmique $\tilde{\mathfrak{f}_{\mathfrak{q}} }$. Donc le symbole de Hasse logarithmique vaut $1$, d'où le résultat.

\end{proof}

\begin{cor}
Si pour tout $\mathfrak{p} \in S$, $\sigma_{\mathfrak{p}}$ appartient au sous-groupe d'inertie logarithmique $\tilde{\Gamma}_{\mathfrak{p}}$ alors via le théorème 2.1.2 précédent, nous pouvons prendre $\alpha_{\mathfrak{p}}$ premier à $\mathfrak{p}$ : ainsi les $ \beta_{\mathfrak{p}}$ seront premiers à $m$ donc au conducteur logarithmique
global de l'extension. Ainsi $\alpha$ est premier au conducteur logarithmique global.

\end{cor}

\bigskip

\section{Interprétation arithmétique du sous-groupe de défaut du principe de Hasse $\ell$-adique}
\bigskip

\subsection{Le théorème et sa preuve}
\noindent Cette section est consacrée à la preuve du théorème suivant :

\begin{theorem}

\noindent Soit $L/K$ une $\ell$-extension abélienne finie, 

\noindent $\tilde{\mathcal{C}}\ell_{L} $ (respectivement $\tilde{\mathcal{C}}\ell_{K} $  ) le groupe des classes logarithmiques de degré nul de $L$ (respectivement $K$) défini par Jaulent \cite{Ja2} comme le quotient du groupe des diviseurs logarithmiques de degré nul par les diviseurs logarithmiques principaux,  

\noindent $\tilde{\mathcal{C}}\ell_{L}^{*}$  le noyau de l'application norme $N_{L/K} : \tilde{\mathcal{C}}\ell_{L} 
\longrightarrow N_{L/K} \tilde{\mathcal{C}}\ell_{L}  $ 

\noindent $\Delta_{L/K}$ l'idéal d'augmentation du groupe de Galois de $L/K$. 

\noindent Nous avons alors :

$$ | \hat{\Gamma}_{L/K}|  (\mathcal{N}_{L/K} :  N_{L/K}\mathcal{R}_{L} ) =(\tilde{\mathcal{C}\ell}_{L}^{*}: \tilde{\mathcal{C}\ell}_{L}^{\Delta_{L/K}}) 
(\tilde{\mathcal{E}}_{K}: \tilde{\mathcal{E}}_{K} \cap  N_{L/K}\mathcal{R}_{L})$$

\end{theorem}
\medskip

\subsubsection{\'Etapes intermédiaires}

\noindent \textit{Remarque I :}

\noindent Soit $L/K$ une $\ell$-extension abélienne finie, $\alpha \in \mathcal{R}_{K}$ étant supposé partout (i.e pour toute place de $K$) norme locale dans $L/K$. 

\noindent Si $\psi(\alpha) = \mathfrak{p}^{\mathfrak{a}} a $ avec $\mathfrak{p} \not\vert a$ et $\mathfrak{a} \in \mathbb{Z}_{\ell}$,
alors $\mathcal{R}_{K}$ étant plongé dans $\mathcal{J}_{K}$ par l'injection diagonale,
 nous en déduisons localement  : 
$$\alpha= \tilde{\pi}_{\mathfrak{p}}^{\mathfrak{a}} u$$
\noindent avec $ \tilde{\pi}_{\mathfrak{p}}$ l'uniformisante logarithmique pour $\mathfrak{p}$ et $u$ une unité logarithmique, et $\tilde{v}_{\mathfrak{p}}(\alpha)=\mathfrak{a}$.

\noindent Or par hypothèse, $\alpha$ est supposé norme locale partout, nous avons donc aussi :
$$\alpha=N_{L_{\mathfrak{P}}/K_{\mathfrak{p}}} ( \tilde{\pi}_{\mathfrak{P}}^{\lambda}  w) $$
pour $\mathfrak{P} \vert \mathfrak{p}$, avec $ \tilde{\pi}_{\mathfrak{P}}$ uniformisante logarithmique dans $L_{\mathfrak{P}}$ et $w$ une unité logarithmique
dans $L_{\mathfrak{P}}$. Or $N_{L_{\mathfrak{P}}/K_{\mathfrak{p}}} ( \tilde{\pi}_{\mathfrak{P}})$
est de la forme $ {\tilde{\pi}_{\mathfrak{p}}^{\tilde{f}_{\mathfrak{p}}} } v$ avec $v$ une unité logarithmique de $K_{\mathfrak{p}}$, il en résulte donc : $\mathfrak{a}= \lambda \tilde{f}_{\mathfrak{p}}$, ainsi 
$\mathfrak{p}^{\mathfrak{a}}= \mathfrak{p}^{ \lambda \tilde{f}_{\mathfrak{p}} }= N_{L/K}(\mathfrak{P}^{\lambda})$,
à savoir $\mathfrak{p}^{a}$ est une norme globale de sorte que $\psi(\alpha)$ est bien une norme globale i.e il existe $\mathfrak{U} \in D\ell_{L/K}$ tel que $$\psi(\alpha)=N_{L/K}(\mathfrak{U}).$$

\noindent Nous pouvons même modifier $\alpha$ modulo $N_{L/K}(x)$ pour $x \in \mathcal{R}_{L}$ de sorte que 
$$ \psi(\alpha N_{L/K}(x))=N_{L/K}(\mathfrak{B}) $$
\noindent avec $\mathfrak{B} \in D\ell_{L}^{\mathfrak{m}}$ pour tout module $\mathfrak{m}$.

\smallskip

\begin{lemma}
$L/K$ étant une $\ell$-extension abélienne finie, $\mathfrak{m}$ un module donné de $K$, $\mathfrak{U} \in D\ell_{L}$ (respectivement $x \in L$) tel que $N_{L/K}(\mathfrak{U})$ soit premier à $\mathfrak{m}$ (respectivement $N_{L/K}(x)$ premier à $\mathfrak{m}$) alors il existe un diviseur logarithmique $\mathfrak{B}$ de $L$  (respectivement $y \in L$ )  premier à $\mathfrak{m}$ tel que
$N_{L/K}(\mathfrak{B})=N_{L/K}(\mathfrak{U})$ (respectivement $N_{L/K}(y)=N_{L/K}(x)$).

\end{lemma}
\smallskip

\smallskip

\noindent \textit{Remarque II :}

\noindent $L/K$ étant une $\ell$-extension abélienne finie, notons
$\mathcal{N}_{L/K}$ les éléments de $\mathcal{R}_{K}$ qui sont partout normes locales,
$\mathcal{N}_{L/K, \mathfrak{m}}= \mathcal{N}_{L/K} \cap \Lambda_{\mathfrak{m}}$, où $\mathfrak{m}$ désigne un module de $K$ et $\Lambda_{\mathfrak{m}}$ les idèles
principaux dont le diviseur logarithmique associé est premier au module $\mathfrak{m}$.
Nous considérons alors l'application canonique :
$$ \mathcal{N}_{L/K, \mathfrak{m}} \longrightarrow \frac{\mathcal{N}_{L/K} } {N_{L/K} \mathcal{R}_{L}}.$$
Le noyau correspond aux éléments $\alpha \in \mathcal{N}_{L/K, m}$ pour lesquels $\exists x \in \mathcal{R}_{L} \; \textrm{tel que} \;  \alpha=N_{L/K}(x)$.  
Comme par hypothèse $\alpha \in \Lambda_{m}$, par application du lemme ci-dessus, nous avons $\alpha \in N_{L/K} \Lambda_{m}$.

\noindent Quant à la surjectivité de l'application ci-dessus : soit $\alpha \in \mathcal{N}_{L/K}$ et soit $\mathfrak{p} \vert m$ divisant 
$\psi(\alpha)$ le diviseur logarithmique associé à $\alpha$ à une certaine puissance $\lambda$ positive, comme $\alpha$ est une norme locale $\lambda$ est divisible par le degré résiduel logarithmique $\tilde{f}_{\mathfrak{p}}$. D'après la remarque I, nous pouvons modifier $\alpha$ modulo $N_{L/K}\mathcal{R}_{L}$ de telle sorte que $\alpha$ soit premier à $m$.

\noindent Ainsi, il vient :
$$\frac{\mathcal{N}_{L/K,m}} {N_{L/K} \Lambda_{m}} \simeq \frac{\mathcal{N}_{L/K} } {N_{L/K} \mathcal{R}_{L}}.$$

\bigskip\bigskip

\begin{defi}
Soit $L/K$ une $\ell$-extension abélienne finie, nous posons :

$$ \mathcal{A}_{L/K} =\{ \alpha \in \mathcal{R}_{K} \; / \; \psi(\alpha) \in N_{L/K}(D\ell_{L}^{\tilde{\mathfrak{f}}_{L/K}}) \}.$$

\noindent Ainsi les éléments de $\mathcal{A}_{L/K}$ ont un diviseur logarithmique associé premier au
conducteur logarithmique global.
\smallskip

\noindent Nous définissons également l'application suivante pour un ensemble $S$ des places de $K$ logarithmiquement ramifiées dans $L$ :

 \begin{center}
$\begin{array}{ccccc}
 \psi_{L/K} & : & \mathcal{A}_{L/K} & \to & \widetilde{\Gamma}_{L/K}= \prod_{\mathfrak{p} \vert \tilde{\mathfrak{f}}_{L/K} } \widetilde{\Gamma}_{L/K, \mathfrak{p}}  \\
 & & \alpha & \mapsto & (....,\widetilde{(\frac{\alpha, L/K}{\mathfrak{p}})} , ...)_{\mathfrak{p} \vert \tilde{\mathfrak{f}}_{L/K} }\\
\end{array}$
\end{center}

\noindent où $\widetilde{\Gamma}_{L/K, \mathfrak{p}}$ désigne le sous-groupe d'inertie logarithmique en $\mathfrak{p}$
dans $L/K$. 

\end{defi}
\smallskip

\noindent \textit{Remarques : }

\noindent Cette application a un sens : car si nous restreignons le symbole de Hasse logarithmique aux idèles principaux
dont le diviseur logarithmique associé est premier à $\mathfrak{p}$, alors l'image du symbole est le sous-groupe d'inertie
logarithmique (théorème 2.1.2).
\medskip

\begin{proposition}
 Nous avons : $\mathrm{Ker}\psi_{L/K}=\{ \alpha \in \Lambda_{\tilde{\mathfrak{f}}_{L/K}} \; \textrm{partout norme locale}  \}$.

\end{proposition}

\begin{proof}

\noindent Compte-tenu de la caractérisation du noyau du symbole de Hasse logarithmique, les éléments du noyau de 
$\psi_{L/K}$ sont ceux qui sont normes locales en tout $\mathfrak{p}$, $\mathfrak{p} \vert \tilde{\mathfrak{f}}_{L/K}$.
Or si $\mathfrak{q} \not \vert \tilde{\mathfrak{f}}_{L/K}$, $\alpha \in \mathcal{A}_{L/K}$ est norme locale en
$\mathfrak{q}$ car $\psi(\alpha)$ est la norme d'un diviseur logarithmique, donc $\alpha$ a la bonne valuation en $\mathfrak{q}$.

\smallskip

\noindent Réciproquement soit $\alpha \in\Lambda_{\tilde{\mathfrak{f}}_{L/K}} $ supposé partout norme locale. Alors nous obtenons en utilisant le remarque II : $\psi(\alpha) \in N_{L/K}(D\ell_{L}^{\tilde{\mathfrak{f}}_{L/K}})$. Ainsi $\alpha \in \mathcal{A}_{L/K}$ et $\alpha$ est norme local en particulier en tout $\mathfrak{p} \vert \tilde{\mathfrak{f}}_{L/K}$. 

\noindent Le noyau est donc bien celui annoncé.

\end{proof}

\smallskip

\begin{defi}
$L/K$ étant une $\ell$-extension abélienne finie, nous posons :
$$ \mathcal{N}_{L/K}= \{ \alpha \in \mathcal{R}_{K} \; \textrm{ qui sont partout norme locale} \}. $$
\noindent Et pur tout module $\mathfrak{m}$ de $K$, posons :
$$ \mathcal{N}_{L/K, \mathfrak{m}} = \mathcal{N}_{L/K} \cap \Lambda_{\mathfrak{m}} .$$
\end{defi}
\smallskip

\begin{defi}
Nous appelons $\widehat{\Gamma}_{L/K}$ le sous-groupe de $\widetilde{\Gamma}_{L/K}= \prod_{\mathfrak{p} \vert \tilde{\mathfrak{f}}_{L/K} } \widetilde{\Gamma}_{L/K, \mathfrak{p}} $ formé des familles $(\sigma_{\mathfrak{p}})_{\mathfrak{p} \vert \tilde{\mathfrak{f}}_{L/K} }$ telles que $\prod_{\mathfrak{p} \vert \tilde{\mathfrak{f}}_{L/K} } \sigma_{\mathfrak{p}}=1$. Nous notons $\widehat{e}_{L/K}$ l'ordre de ce sous-groupe.

\end{defi}
\smallskip

\begin{theorem}
$L/K$ étant une $\ell$-extension abélienne finie, alors nous avons la suite exacte suivante :

$$ 1 \longrightarrow \mathcal{N}_{L/K, \tilde{\mathfrak{f}}_{L/K}   }     \longrightarrow \mathcal{A}_{L/K}           \longrightarrow  \widehat{\Gamma}_{L/K}     \longrightarrow 1    $$

\end{theorem}
\medskip

\begin{proof}
C'est dû d'une part à la réciproque de la formule du produit pour le symbole de Hasse logarithmique.
Nous appliquons en effet le théorème 2.2.2 pour $S$ l'ensemble des places de $K$ logarithmiquement ramifiées dans $L$ et $y=(\sigma_{\mathfrak{p}})_{\mathfrak{p} \vert \tilde{\mathfrak{f}}_{L/K} } \in \widehat{\Gamma}_{L/K}  $.
Nous pouvons donc expliciter un antécédent $\alpha \in \mathcal{R}_{K}$ premier au conducteur logarithmique global tel que :
$\widetilde{(\frac{\alpha, L/K}{\mathfrak{p}}) }=\sigma_{\mathfrak{p}}$, i.e dans $\mathcal{A}_{L/K}$. $\widehat{\Gamma}_{L/K}  $ est bien l'image décrite. D'autre part, le noyau de $\psi_{L/K}$ a été décrit précédemment dans la proposition 3.1.1.

\end{proof}

\subsubsection{La démonstration du théorème}
\medskip

\begin{proof}
Soit $\tilde{f}_{L/K}$ le conducteur logarithmique global
de l'extension $L/K$ considérée et
soit $\alpha \in \mathcal{A}_{L/K}$, alors par définition nous avons :
$$\psi(\alpha)=N_{L/K}(\mathcal{U})$$ 
\noindent avec $\mathcal{U} \in D\ell_{L/K}^{\tilde{f}_{L/K}}$. Or $\mathcal{U}$ est défini modulo l'idéal d'augmentation
noté $\Delta_{L/K}$,

\noindent $\mathcal{U} \in \frac{ D\ell_{L/K}^{\tilde{f}_{L/K}} } { {D\ell_{L/K}^{\tilde{f}_{L/K}}}^{\Delta_{L/K}}}$.
Nous considérons alors $c\ell(\mathcal{U})$ la classe de $\mathcal{U}$ modulo $\tilde{\mathcal{C}\ell}_{L}^{\Delta_{L/K}}.$

\noindent Nous obtenons ainsi une application :
$$ \mathcal{A}_{L/K} \longrightarrow \frac{\tilde{\mathcal{C}\ell}_{L}^{*}}{\tilde{\mathcal{C}\ell}_{L}^{\Delta_{L/K}}}. $$

\noindent Nous explicitons le noyau de cette application :

\noindent $c\ell(\mathcal{U}) \in \tilde{\mathcal{C}\ell}_{L}^{\Delta_{L/K}} $ si et seulement si
$\mathcal{U} =\psi(x) \prod \mathfrak{P}_{i}^{\sigma_{i}-1}$ 
avec $x \in \mathcal{R}_{L}$, $\sigma_i \in \textrm{Gal}(L/K)$ et $\mathfrak{P}_{i} \in \Delta_{L/K}$.
Ainsi $\psi(\alpha)=N_{L/K}(\mathcal{U})=N_{L/K}(\psi(x))=\psi(N_{L/K}(x))$,
i.e $\alpha$ et $x$ ne différent que par une unité logarithmique $\epsilon \in \widetilde{\mathcal{E}}_{K}$ d'où :
$\alpha= N_{L/K}(x) \epsilon$.
Comme $\psi(\alpha)=\psi(N_{L/K}(x))$, nous en déduisons le noyau :
$\widetilde{\mathcal{E}}_{K} N_{L/K}\Lambda_{\tilde{f}_{L/K}}(L)$
\noindent où $\Lambda_{\tilde{f}_{L/K}}(L)$ désigne les idèles principaux de $L$ dont le diviseur
logarithmique associé est premier au conducteur logarithmique.

\noindent De plus cette application est surjective.

\noindent Nous obtenons donc la suite exacte suivante :

$$1 \longrightarrow \widetilde{\mathcal{E}}_{K} N_{L/K}\Lambda_{\tilde{f}_{L/K}} \longrightarrow \mathcal{A}_{L/K} 
\longrightarrow       \frac{\tilde{\mathcal{C}\ell}_{L}^{*}}{\tilde{\mathcal{C}\ell}_{L}^{\Delta_{L/K}}}
\longrightarrow 1$$

\noindent Or nous disposons de la suite exacte :

$$ 1 \longrightarrow  \mathcal{N}_{L/K, \tilde{f}_{L/K}} \longrightarrow \mathcal{A}_{L/K} \longrightarrow
\hat{\Gamma}_{L/K} \longrightarrow 1.$$

\noindent Et nous avons

 $N_{L/K}\Lambda_{\tilde{f}_{L/K}} \subset  \mathcal{N}_{L/K, \tilde{f}_{L/K}}
\cap \widetilde{\mathcal{E}}_{L/K} N_{L/K}\Lambda_{\tilde{f}_{L/K}}.$

\noindent Ainsi il vient,

$$\frac{\mathcal{A}_{L/K}  } {   \mathcal{N}_{L/K, \tilde{f}_{L/K}} } \simeq \hat{\Gamma}_{L/K}.$$

\noindent Donc via les théorèmes d'isomorphisme, nous obtenons :
 $$  \frac {  \frac{\mathcal{A}_{L/K}  } {   N_{L/K}\Lambda_{\tilde{f}_{L/K}} }   }
  {\frac {\mathcal{N}_{L/K, \tilde{f}_{L/K}}} { N_{L/K}\Lambda_{\tilde{f}_{L/K}} } } \simeq \hat{\Gamma}_{L/K}$$

\noindent il en découle :

$$ | \hat{\Gamma}_{L/K}|  |\frac {\mathcal{N}_{L/K, \tilde{f}_{L/K}}} { N_{L/K}\Lambda_{\tilde{f}_{L/K}} } | =| \frac{\mathcal{A}_{L/K}  } {   N_{L/K}\Lambda_{\tilde{f}_{L/K}} }  |.$$

\noindent L'autre suite exacte nous indique :

$$ \frac{\mathcal{A}_{L/K}  } {   \tilde{\mathcal{E}}_{K} N_{L/K}\Lambda_{\tilde{f}_{L/K}} } \simeq 
 \frac{\tilde{\mathcal{C}\ell}_{L}^{*}}{\tilde{\mathcal{C}\ell}_{L}^{\Delta_{L/K}}} $$

\noindent et à nouveau via le théorème du double quotient, nous obtenons comme $   N_{L/K}\Lambda_{\tilde{f}_{L/K}}      \subset   \tilde{\mathcal{E}}_{K} N_{L/K}\Lambda_{\tilde{f}_{L/K}}$ :

$$\frac
{  \frac {\mathcal{A}_{L/K} } { N_{L/K}\Lambda_{\tilde{f}_{L/K}} }  } 
{ \frac {  \tilde{\mathcal{E}}_{K} N_{L/K}\Lambda_{\tilde{f}_{L/K}} }  { N_{L/K}\Lambda_{\tilde{f}_{L/K}}} } 
\simeq 
 \frac{\tilde{\mathcal{C}\ell}_{L}^{*}}{\tilde{\mathcal{C}\ell}_{L/K}^{\Delta_{L}}} .$$

\noindent Finalement, nous obtenons :

$$ |  \frac {\mathcal{A}_{L/K} } { N_{L/K}\Lambda_{\tilde{f}_{L/K}} } | = 
|  \frac {  \tilde{\mathcal{E}}_{K} N_{L/K}\Lambda_{\tilde{f}_{L/K}} }  { N_{L/K}\Lambda_{\tilde{f}_{L/K}}} |
| \frac{\tilde{\mathcal{C}\ell}_{L}^{*}}{\tilde{\mathcal{C}\ell}_{L}^{\Delta_{L/K}}} |. $$

\noindent Or 

$$ \frac {  \tilde{\mathcal{E}}_{K} N_{L/K}\Lambda_{\tilde{f}_{L/K}} }  { N_{L/K}\Lambda_{\tilde{f}_{L/K}}} \simeq 
\frac {\tilde{\mathcal{E}}_{K} } {\tilde{\mathcal{E}}_{L/K} \cap  N_{L/K}\Lambda_{\tilde{f}_{L/K}} }.$$

\noindent Finalement, nous avons :

$$ | \hat{\Gamma}_{L/K}|  (\mathcal{N}_{L/K, \tilde{f}_{L/K}} :  N_{L/K}\Lambda_{\tilde{f}_{L/K}}  ) =(\tilde{\mathcal{C}\ell}_{L}^{*}: \tilde{\mathcal{C}\ell}_{L}^{\Delta_{L/K}}) 
(\tilde{\mathcal{E}}_{K}: \tilde{\mathcal{E}}_{K} \cap  N_{L/K}\Lambda_{\tilde{f}_{L/K}})$$

\noindent d'où le résultat compte-tenu de la remarque II.

\end{proof}
\medskip

\subsection{Exemple d'application du théorème}

\noindent \textbf{Exemple :}
\smallskip

\noindent Considérons l'extension biquadratique suivante : $L=\mathbb{Q}(i, \sqrt{7})$.
Nous avons alors le schéma suivant :

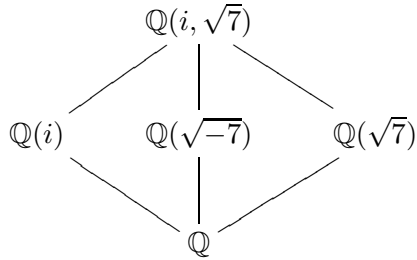
\[
 \xymatrix{
       & \mathbb{Q}(i,\sqrt7)  \ar@{-}[d] \ar@{-}[ld] \ar@{-}[rd]&  \\
   \mathbb{Q}(i)\ar@{-}[rd] & \mathbb{Q}(\sqrt{-7})\ar@{-}[d] & \mathbb{Q}(\sqrt7) \ar@{-}[ld] \\
      & \mathbb{Q} &    
    }
\]

\bigskip

\noindent{Bilan des places ramifiées au sens de la ramification classique : }

-$2$ est  ramifiée dans $\mathbb{Q}(i)/\mathbb{Q}$
\smallskip

-$2$ et $7$ sont ramifiées dans $\mathbb{Q}(\sqrt{-7})/\mathbb{Q}$
\smallskip

-$2$ et $7$ sont ramifiées dans $\mathbb{Q}(\sqrt7)/\mathbb{Q}$
\bigskip

\noindent{Bilan des places ramifiés au sens de la ramification logarithmique : }

\noindent Nous rappelons que dans le cadre des $\ell$-extensions ramification
classique et logarithmique ne diffèrent que pour les places au dessus de $\ell$ \cite{Ja2}.
Nous utilisons également \cite{Re2}.
\smallskip

-$2$ est logaritmiquement ramifiée dans $\mathbb{Q}(i)/\mathbb{Q}$, 

\indent $7$ y est logarithmiquement non ramifiée 
\smallskip

-$2$ est logarithmiquement non ramifiée dans $\mathbb{Q}(\sqrt{-7})/\mathbb{Q}$ (car $-7 \equiv 1 \; \textrm{mod} 8$),

\indent $7$ y est logarithmiquement ramifiée
\smallskip

-$2$ est logarithmiquement ramifiée dans $\mathbb{Q}(\sqrt7)/\mathbb{Q}$, tout comme $7$. 
\bigskip

\noindent Calcul des différents termes de l'expression du défaut du symbole de Hasse logarithmique :

-pour $ | \hat{\Gamma}_{L/\mathbb{Q}}| $ :

\noindent le sous-groupe d'inertie logarithmique fixe par définition l'extension maximale logarithmiquement non ramifiée.
Nous avons :
$\tilde{\Gamma}_{L,2}=\textrm{Gal}(L/\mathbb{Q}(\sqrt 7))= <\sigma>$ et $ \tilde{\Gamma}_{L,7}=\textrm{Gal}(L/\mathbb{Q})(i)= <\tau>$. Il en découle $\hat{e}_{L/\mathbb{Q}}=2$. 

\smallskip

-pour $ (\tilde{\mathcal{C}\ell}_{L}^{*}: \tilde{\mathcal{C}\ell}_{L}^{\Delta_{L/\mathbb{Q}}}) $ :

\noindent $L$ est logarithmiquement principal (\cite[p.45]{So1}) donc ce terme est trivial.
\smallskip

-pour $ (\tilde{\mathcal{E}}_{\mathbb{Q}}: \tilde{\mathcal{E}}_{\mathbb{Q}} \cap  N_{L/\mathbb{Q}}\mathcal{R}_{L})$ :

\noindent d'après \cite[Rem p.3]{So2}, $\tilde{\mathcal{E}}_{\mathbb{Q}}$ coincide avec le tensorisé du groupe des $2$-unités logarithmiques de $\mathbb{Q}$, et est en particulier généré par $-1$ et $2$. Or $-1$ n'est pas norme.
\'Etudions le cas $2$  : comme $7 \equiv 1 \; \textrm{mod} 8$,  $-d$ est une unité $2$-adique, ainsi
$\mathbb{Q}_{2}(i)=\mathbb{Q}_{2}(\sqrt d)$ d'où $[\mathbb{Q}_{2}(i, \sqrt d):\mathbb{Q}_{2}]=2$.
Notons $\mathfrak{P}$ une place de $L$ au dessus de $2$. Comme $2$ est logarithmiquement ramifiée dans $L$,
$L_{\mathfrak{P}}  \cap \mathbb{Q}_{2}^{c} \subset L_{\mathfrak{P}}$, cette inclusion étant stricte,
il vient donc : $L_{\mathfrak{P}}  \cap \mathbb{Q}_{2}^{c}= \mathbb{Q}_{2}$. L'indice de ramification logarithmique est donc trivial, ainsi $2$ est norme : $N_{}(\mathfrak{P})=1.2=2$.
Il s'ensuit :  $ (\tilde{\mathcal{E}}_{\mathbb{Q}}: \tilde{\mathcal{E}}_{\mathbb{Q}} \cap  N_{L/\mathbb{Q}}\mathcal{R}_{L})=2.$ 
\smallskip

\noindent Par application précédent, nous en déduisons que dans cet exemple, le défaut du symbole de Hasse logarithmique
est trivial : $(\mathcal{N}_{L/\mathbb{Q}, \tilde{f}_{L/\mathbb{Q}}} :  N_{L/K} \mathcal{R}_{\mathbb{Q}}  )=1$.
Le principe de Hasse logarithmique s'applique donc dans cette extension.

\bigskip

\bigskip\bigskip

\subsection{Premier cas particulier : celui des $\ell$-extensions cycliques}
\smallskip

\noindent Rappelons d'abord la formule des classes logarithmiques ambiges  {~\cite[§4]{Ja2}} }.

\noindent Par définition du groupe des classes ambiges, la cohomologie nous donne le diagramme suivant :
\smallskip

$$ \xymatrix{
    1 \ar[r]  & \tilde{P\ell}_{K}  \ar[r] \ar[d] & \tilde{D\ell}_{K}  \ar[r] \ar[d] & \tilde{\mathcal{C}\ell}_{K} \ar[r] \ar[d]_{\tilde{j}}&1 \\
    1  \ar[r]  & \tilde{P\ell}_{L}^{G} \ar[r]  & \tilde{D\ell}_{L}^{G}  \ar[r]  & \tilde{\mathcal{C}\ell}_{L}^{G}  \ar[r]  & \mathrm{H^{1}}(G,\tilde{P\ell}_{L} ) \ar[r]  &  \mathrm{H^{1}}(G,\tilde{D\ell}_{L} ) }$$

\smallskip

\noindent où $\tilde{j}$ est le morphisme d'extension de classes. Le lemme du Serpent nous permet alors d'écrire la suite :
$$ 1 \longrightarrow \tilde{P\ell}_{L}^{G}/\tilde{P\ell}_{K} \longrightarrow \tilde{D\ell}_{L}^{G}/\tilde{D\ell}_{K} \longrightarrow
\tilde{\mathcal{C}\ell}_{L}^{G}/ \tilde{j}(\tilde{\mathcal{C}\ell}_{K}) \longrightarrow \mathrm{H^{1}}(G,\tilde{P\ell}_{L} )
\overset{\phi}{ \longrightarrow} \mathrm{H^{1}}(G,\tilde{D\ell}_{L} ).$$

\noindent puis la formule :
$$  | \tilde{\mathcal{C}\ell}_{L}^{G}|=| \tilde{\mathcal{C}\ell}_{K}| \frac{(\tilde{D\ell}_{L}^{G}:\tilde{D\ell}_{K}) \mathrm{H^{1}}(G,\tilde{P\ell}_{L} )} {(\tilde{P\ell}_{L}^{G}:\tilde{P\ell}_{K}) \mathrm{H^{1}}(G,\tilde{D\ell}_{L} )}              |\mathrm{Coker} \phi | $$
\medskip

\begin{proposition}{Formule des classes logarithmiques ambiges {~\cite[§4]{Ja2}} }

\noindent $L/K$ étant une $\ell$-extension cyclique de groupe de Galois $G$, satisfaisant la conjecture de Gross,
 l'ordre du sous-groupe ambige $\tilde{\mathcal{C}\ell}_{L}^{G}$ est :

$$ | \tilde{\mathcal{C}\ell}_{L}^{G}|= \frac {|\tilde{\mathcal{C\ell}}_{K}|  \prod_{\mathfrak{p} \in P\ell_{K}^{\infty}} 
d_{\mathfrak{p}}(L/K) \prod_{\mathfrak{p} \in P\ell_{K}^{0}} \tilde{e}_{\mathfrak{p}}(L/K) }
{[L^{c}:K^{c}] (\tilde{\mathcal{E}}_{K} :\tilde{\mathcal{E}}_{K} \cap N_{L/K} \mathcal{R}_{L} )} |\mathrm{Coker} \phi |$$

\end{proposition}

\bigskip

\begin{proposition}

\noindent $L/K$ étant une $\ell$-extension cyclique de groupe de Galois $G$, satifsaisant la conjecture de Gross, nous avons :
$$ (\tilde{\mathcal{C}\ell}_{K} : N_{L/K} \tilde{\mathcal{C}\ell}_{L}) =\frac{ | \hat{\Gamma}_{L/K}| [L^{c}:K^{c}] }   {\prod_{\mathfrak{p} \in P\ell_{K}^{\infty}} 
d_{\mathfrak{p}}(L/K) \; \prod_{\mathfrak{p} \in P\ell_{K}^{0}}  \tilde{e}_{\mathfrak{p}}(L/K)  \;  |\mathrm{Coker} \phi |}  .       $$

\end{proposition}

\medskip

\begin{proof}

\noindent  $L/K$ étant une $\ell$-extension cyclique, par le théorème $\ell$-adique de
la norme de Hasse \cite{Re1} nous savons : $(\mathcal{N}_{L/K} : N_{L/K}\mathcal{R}_{L})=1$.
Il vient donc en appliquant la formule du théorème précédent :

$$ | \hat{\Gamma}_{L/K}| =(\tilde{\mathcal{C}\ell}_{L}^{*}: \tilde{\mathcal{C}\ell}_{L}^{\Delta_{L/K}}) 
(\tilde{\mathcal{E}}_{K}: \tilde{\mathcal{E}}_{K} \cap  N_{L/K} \mathcal{R}_{L}).$$

\noindent Or la formule des classes ambiges établie par Jaulent nous donne dans ce contexte :

$$ (\tilde{\mathcal{E}}_{K} : \tilde{\mathcal{E}}_{K} \cap  N_{L/K} \mathcal{R}_{L})= \frac {|\tilde{\mathcal{C\ell}}_{K}|  \prod_{\mathfrak{p} \in P\ell_{K}^{\infty} |\mathrm{Coker} \phi |} 
d_{\mathfrak{p}}(L/K) \prod_{\mathfrak{p} \in P\ell_{K}^{0}} \tilde{e}_{\mathfrak{p}}(L/K) }
{[L^{c}:K^{c}]  |\tilde{\mathcal{C}\ell}_{L}^{G}| }  |\mathrm{Coker} \phi |.$$

\noindent Il en découle :
$$ (\tilde{\mathcal{C}\ell}_{L}^{*}: \tilde{\mathcal{C}\ell}_{L}^{\Delta_{L/K}}) \frac{|\tilde{\mathcal{C\ell}}_{K}| } { |\tilde{\mathcal{C}\ell}_{L}^{G}|}  = \frac{ | \hat{\Gamma}_{L/K}| [L^{c}:K^{c}] }   {\prod_{\mathfrak{p} \in P\ell_{K}^{\infty}} 
d_{\mathfrak{p}}(L/K) \; \prod_{\mathfrak{p} \in P\ell_{K}^{0}} \tilde{e}_{\mathfrak{p}}(L/K) \; |\mathrm{Coker} \phi | }              .$$

\noindent Or nous disposons des deux suites exactes suivantes :
$\sigma$ désignant un générateur de $G=\textrm{Gal}(L/K)$,
$$ 1 \longrightarrow     \tilde{\mathcal{C}\ell}_{L}^{*}
\longrightarrow \tilde{\mathcal{C}\ell}_{L}
\longrightarrow N_{L/K} \tilde{\mathcal{C}\ell}_{L}
\longrightarrow 1  \; (1)$$

$$ 1 \longrightarrow   \tilde{\mathcal{C}\ell}_{L}^{G}  
\longrightarrow \tilde{\mathcal{C}\ell}_{L}
\longrightarrow \tilde{\mathcal{C}\ell}_{L}^{\sigma-1}
\longrightarrow 1  \; (2)$$

\noindent D'après (1), nous en déduisons :
$$ |\tilde{\mathcal{C}\ell}_{L}^{*} |= \frac {|\tilde{\mathcal{C}\ell}_{L}|} {|N_{L/K} \tilde{\mathcal{C}\ell}_{L}|}$$

\noindent et par (2), nous avons également :
$$ |\tilde{\mathcal{C}\ell}_{L}^{\sigma-1}|= \frac{|\tilde{\mathcal{C}\ell}_{L}|} {| \tilde{\mathcal{C}\ell}_{L}^{G}|}.$$

\noindent Finalement nous obtenons :
$$ (\tilde{\mathcal{C}\ell}_{L}^{*}: \tilde{\mathcal{C}\ell}_{L}^{\Delta_{L/K}}) \frac{|\tilde{\mathcal{C\ell}}_{K}| } { |\tilde{\mathcal{C}\ell}_{L}^{G}|} = (\tilde{\mathcal{C}\ell}_{K} : N_{L/K} \tilde{\mathcal{C}\ell}_{L} ) .$$

\noindent Ainsi, nous avons :
$$ (\tilde{\mathcal{C}\ell}_{K} : N_{L/K} \tilde{\mathcal{C}\ell}_{L}) =\frac{ | \hat{\Gamma}_{L/K}| [L^{c}:K^{c}] }   {\prod_{\mathfrak{p} \in P\ell_{K}^{\infty}} 
d_{\mathfrak{p}}(L/K) \; \prod_{\mathfrak{p} \in P\ell_{K}^{0}} \; \tilde{e}_{\mathfrak{p}}(L/K) \; |\mathrm{Coker} \phi | }    .       $$

\end{proof}
\medskip

\subsection {Deuxième cas particulier}

\noindent Considérons maintenant une $\ell$-extension cyclique $L/K$ pour laquelle $ | \hat{\Gamma}_{L/K}|=1$. 

\noindent Comme précédemment  le théorème $\ell$-adique de
la norme de Hasse \cite{Re1} nous donne :
$(\mathcal{N}_{L/K} : N_{L/K}\mathcal{R}_{L})=1$.

\noindent Par le théorème 3.1.1, il vient :
$$(\tilde{\mathcal{C}\ell}_{L}^{*}: \tilde{\mathcal{C}\ell}_{L}^{\Delta_{L/K}}) (\tilde{\mathcal{E}}_{K} : \tilde{\mathcal{E}}_{K} \cap  N_{L/K} \mathcal{R}_{L})=1.$$

\noindent Il s'agit alors d'un produit de deux entiers valant $1$, nous obtenons alors les égalités suivantes compte tenu des
relations d'inclusion :
$$ \tilde{\mathcal{C}\ell}_{L}^{*}= \tilde{\mathcal{C}\ell}_{L}^{\Delta_{L/K}} \qquad \tilde{\mathcal{C}\ell}_{L}^{G}=N_{L/K} 
\tilde{\mathcal{C}\ell}_{L} \qquad \tilde{\mathcal{E}}_{K} \subseteq  N_{L/K} \mathcal{R}_{L}.$$

\medskip

\subsection{Corollaire : version logarithmique du théorème de l'idéal principal}

\begin{cor}

\noindent Supposons que $L/K$ soit une $\ell$-extension cyclique logarithmiquement non ramifiée,
satisfaisant la conjecture de Gross, il existe alors des diviseurs logarithmiques de $K$ non principaux qui
deviennent principaux par extension dans $L$.

\end{cor}
\smallskip

\begin{proof}
Soit $\tilde{j}$ le  morphisme d'extension de $ \tilde{\mathcal{C}\ell}_{K}$ dans $ \tilde{\mathcal{C}\ell}_{L}$

\noindent Comme l'extension est supposée cyclique, nous avons $\tilde{j}( \tilde{\mathcal{C}\ell}_{K}) \subseteq \tilde{\mathcal{C}\ell}_{L}^{G}$.

\noindent D'après \cite{Ja2} , nous avons l'inégalité suivante sur sur le groupe des classes logarithmique ambiges pour une $\ell$-extension cyclique :
$$ |\tilde{\mathcal{C}\ell}_{L}^{G}| \leq |\tilde{\mathcal{C}\ell}_{K}| \frac { \prod_{\mathfrak{p} \in P\ell_{K}^{\infty}} 
d_{\mathfrak{p}}(L/K) \prod_{\mathfrak{p} \in P\ell_{K}^{0}} \tilde{e}_{\mathfrak{p}}(L/K) }
{[L^{c}:K^{c}] (\tilde{\mathcal{E}}_{K} :\tilde{\mathcal{E}}_{K} \cap N_{L/K} \mathcal{R}_{L} )}                |\textrm{H}^{1}(G, \tilde{D\ell}_{L})|.$$

\noindent Or ici $L/K$ est supposée logarithmiquement non ramifiée, compte-tenu de la sous-section précédente,
nous obtenons l'inégalité simplifiée :
$$ |\tilde{j}(\tilde{\mathcal{C}\ell}_{K})| \leq |\tilde{\mathcal{C}\ell}_{L}^{G}| \leq \frac{|\tilde{\mathcal{C}\ell}_{K}|} {[L^{c}:K^{c}]}  |\textrm{H}^{1}(G, \tilde{D\ell}_{L})|.$$

\noindent Compte tenu de l'isomorphisme $\frac{\tilde{\mathcal{C}\ell}_{K} } {\textrm{Ker}\tilde{j}} \simeq \tilde{j}(\tilde{\mathcal{C}\ell}_{K})$, nous en déduisons :
$$ |\textrm{Ker}\tilde{j}| \geq \frac{[L^{c}:K^{c}]}{|\textrm{H}^{1}(G, \tilde{D\ell}_{L})|} .$$

\noindent L'isomorphisme établi par Jaulent \cite{Ja2}
$$ \textrm{H}^{1}(G, \tilde{D\ell}_{L}) \simeq \mathbb{Z}_{\ell}/ [L^{c}:K^{c}] \mathbb{Z}_{\ell} + \sum_{\mathfrak{p}}                                                 \frac{[L^{c}:K^{c}]}{\tilde{e}_{\mathfrak{p}} (L/K) }    \frac{ \textrm{deg}_{K}(\mathfrak{p}) } {\textrm{deg}_{K} D\ell_{K}}  \mathbb{Z}_{\ell}.$$

\noindent nous conduit à $$ |\textrm{Ker}\tilde{j}| > 1.$$

\end{proof}

\end{document}